\documentclass[american,a4paper,12pt]{amsart}
\pdfoutput=1
\usepackage[utf8]{inputenc}
\usepackage{mathrsfs}
\usepackage[inline]{enumitem}
\usepackage{mathtools}
\usepackage{amssymb}
\usepackage{graphicx}
\usepackage[unicode=true,pdfusetitle,
 bookmarks=true,bookmarksnumbered=false,bookmarksopen=false,
 breaklinks=false,pdfborder={0 0 1},backref=false,colorlinks=false]{hyperref}
\usepackage{breakurl}
\usepackage{subcaption}
\usepackage{pgfplots}
\pgfplotsset{compat=1.16}
\usepackage[foot]{amsaddr}
\usepackage{cite}

\usepackage[left=80pt, right=80pt]{geometry}

\theoremstyle{plain}
\newtheorem{thm}{Theorem}[section]
\theoremstyle{plain}
\newtheorem*{thm*}{Theorem}
\theoremstyle{plain}
\newtheorem{prop}[thm]{Proposition}
\theoremstyle{plain}

\theoremstyle{plain}
\newtheorem{cor}[thm]{Corollary}
\theoremstyle{remark}
\newtheorem{conj}{Conjecture}[section]
\theoremstyle{remark}
\newtheorem*{rem*}{Remark}
\theoremstyle{remark}
\newtheorem{defi}{Definition}[section]

\usepackage{fourier}

\newcommand{\urlrepo}{https://plmlab.math.cnrs.fr/fdebarre/2021_DriveInSpace}

\begin{document}

\title{Demographic feedbacks can hamper the spatial spread of a gene drive}

\author{L\'{e}o Girardin}
\address[L.G.]{CNRS, Institut Camille Jordan, Universit\'{e} Claude Bernard Lyon-1, 43 boulevard du 11 novembre 1918, 69622 Villeurbanne Cedex, France}
\email{leo.girardin@math.cnrs.fr}

\author{Florence D\'{e}barre}
\address[F. D.]{CNRS, Sorbonne Universit\'{e}, Univ. Paris Est Creteil, INRAE, IRD, Institute of Ecology and Environmental sciences - Paris, IEES-Paris, 4 place Jussieu, 75005 Paris, France}
\email{florence.debarre@normalesup.org}

\begin{abstract}
    This paper is concerned with a reaction--diffusion system modeling the fixation and the invasion in a population of a gene drive 
    (an allele biasing inheritance, increasing its own transmission to offspring). 
    In our model, the gene drive has a negative effect on the fitness of individuals carrying it, and is therefore susceptible of decreasing the total carrying capacity of the population locally in space. This tends to generate an opposing demographic advection that the gene drive has to overcome in order to invade. While previous reaction--diffusion models neglected this aspect, here we focus on it and try to predict the sign of the traveling wave speed. It turns out to be an analytical challenge, only partial results being within reach, and we complete our theoretical analysis by numerical simulations. Our results indicate that taking into account the interplay between population dynamics and population genetics might actually be crucial, as it can effectively reverse the direction of the invasion and lead to failure. Our findings can be extended to other bistable systems, such as the spread of cytoplasmic incompatibilities caused by Wolbachia.
\end{abstract}


\maketitle

\section{Introduction}
\subsection{Biological background}

Some genetic elements can bias inheritance in their favor, and therefore spread in a population over successive generations, even if they are costly to the organism carrying them \cite{werren_selfish_2011}, a phenomenon called gene drive \cite{alphey_opinion_2020}. Inspired from natural systems, synthetic gene drives have been developed over the last decades for the control of natural populations \cite{sinkins_gene_2006}. Artificial gene drives can be used to spread a genetic modification in a natural population \cite{champer_cheating_2016}, in order to \begin{enumerate*}[label=\textit{(\roman*)}]\item modify this target population (for instance, make disease vectors resistant to the pathogen that they transmit) without significantly affecting its size (replacement drives), or \item reduce the size of the target population by spreading an allele that affects fecundity or survival (suppression drives), or even eradicate the target population (eradication drives) \cite{Dhole_Lloyd_Gould_2020}\end{enumerate*}. 

One way of biasing transmission involves copying the drive element onto the other chromosome in heterozygous cells \cite{burt_site-specific_2003}, a process called gene conversion or ``homing''. A target sequence on the other chromosome is recognized, cleaved, and a DNA-repair mechanism using the intact chromosome as a template leads to the duplication of the drive element. The drive element is then present in most gametes and therefore in most offspring.  
The high flexibility and programmability brought about by the CRISPR-Cas genome editing technology greatly simplifies the creation of this type of artificial gene drives \cite{Esvelt_2014}, and has led to a renewed interest in this technique of population control.  

To date, artificial gene drives have been confined to laboratory experiments; field trials, whether small-scale and confined, or open, have not taken place yet \cite{efsa_panelon_genetically_modified_organisms_gmo_adequacy_2020}. In between, mathematical models can help identify features of gene drives that are most important in determining their dynamics; models can also help anticipate potential issues, and are an essential step between lab and field experiments \cite{NASEM_2016}. 

A key feature affecting the spread of gene drives is the existence or not of a release threshold. Some drive systems spread from arbitrary low introduction frequencies (``thres\-hold-independent drives''), while some others only do if enough drive-carrying individuals are introduced in the population (``threshold-dependent drives'' or ``high-threshold drives'') \cite[p19]{efsa_panelon_genetically_modified_organisms_gmo_adequacy_2020}. In mathematical terms, a release threshold corresponds to the existence of an unstable internal equilibrium (\textit{i.e.}, where wild-type and drive alleles coexist), in other words, of a bistability. Homing-based gene drives can belong to the two categories, depending on their intrinsic parameters \cite{Deredec_2008}. 

The existence of a release threshold affects the spatial spread of a drive and the potential for spatial confinement. High-threshold drives can be spatially confined if dispersal is limited \cite{altrock_using_2010, marshall_confinement_2012, greenbaum_designing_2019}. In spatially continuous environment, the shape of the wave of advance of a drive system changes from pulled for no-threshold drives to pushed in high-threshold drives \cite{Dhole_Lloyd_Gould_2020}. A wave of advance of a high-threshold drive can be stopped by obstacles \cite{Tanaka_Stone_N}. More generally, these results are in line with those of models of spatial spread of bistable systems \cite{barton_dynamics_1979, Barton_Turelli}.       

Many mathematical models of the spatial spread of gene drives focus on allele frequencies and ignore changes in population density. Changes in the size of the target population will however happen with the spread of a gene drive, either as a potential side effect for replacement drives, or as a feature for suppression and eradication drives. Considering these spatial and temporal variations in population size matters, because they create fluxes of individuals from more densely to less densely populated locations, opposing the spread of the gene drive \cite{Dhole_Lloyd_Gould_2020}. 

Beaghton \textit{et al.} \cite{beaghton_gene_2016} used a reaction--diffusion framework to study the spatial spread of a driving-Y chromosome causing population suppression or eradication, and explicitly followed population densities. In a driving-Y system, the offspring of driving-Y bearing males are almost all driving-Y bearing males, because the development of X-chromosome bearing gametes is disrupted. A driving-Y is a no-threshold drive: in a well-mixed population, a driving-Y system can increase from arbitrary low frequencies. Beaghton \textit{et al.} showed that a driving-Y system would spread spatially, and they calculated the speed of the wave of advance.

In this article, we use a system of partial differential equations to study the spread of a homing-based gene drive over a one-dimensional space, explicitly taking into account changes in population sizes. In our model, the fitness cost $s$ associated to the drive affects the level of population suppression, but also the existence and value of the release threshold. We explore the model numerically and prove some results mathematically for large regions of the parameter space. We explore the robustness of our finding by considering other forms of density dependence (including Allee effects), assuming fitness effects on other fitness components, and finally extend our result to another bistable system, the spread of cytoplasmic incompatibilities brought about by \textit{Wolbachia}.





\subsection{Organization of the paper}
In the rest of Section 1, we present our mathematical model, relate it to the literature,
and state our results. In Section 2, we explain our
numerical method and prove our analytical results. In Section 3, we conclude with a mathematical and biological discussion.

\subsection{Derivation of the reaction--diffusion system}

To derive our equations, we follow the methodology presented in \cite{Rode_Debarre_2020}. In addition, we assume that heterozygous individuals are functionally equivalent to drive homozygotes. This can be the case when the conversion rate is $100\%$, \textit{and} either \begin{enumerate*}[label=\textit{(\roman*)}]\item gene conversion (also called homing) takes place early in development (typically in the zygote -- as opposed to gene conversion in the germline) -- so that heterozygotes are all immediately converted into homozygotes, and no heterozygous individuals are ever introduced to the population from an outside source, or \item gene conversion takes place later in the life cycle, but the drive is dominant, so that heterozygotes have the same fitness as drive homozygotes (and then only transmit the drive allele, since gene conversion is $100\%$)\end{enumerate*}. This assumption simplifies a lot the model since we have to track two genotypes instead of three in general: in scenario \textit{i}), there are never heterozygous individuals in the population, and in scenarion \textit{ii}), heterozygotes can be conflated with drive homozygotes since they behave exactly the same way. 
This leads to the following reaction--diffusion system:
\begin{equation}
    \begin{cases}
	\partial_t n_D -\sigma_D\Delta n_D = \omega_{D}n_D\beta_{D}B(n_D+n_O)\frac{2\beta_{O}n_O+\beta_{D}n_D}{n_D+n_O}-d_{D}D(n_D+n_O)n_D, \\
	\partial_t n_O -\sigma_O\Delta n_O = \omega_{O}n_O\beta_{O}B(n_D+n_O)\frac{\beta_{O}n_O}{n_D+n_O}-d_{O}D(n_D+n_O)n_O,
    \end{cases}
    \label{sys:initial_full_system}
\end{equation}
where: 
\begin{itemize}
    \item $n_D(t,x)$ and $n_O(t,x)$, nonnegative functions of time $t$ and space $x$,
	are (continuous densities approximating) the number of alleles $D$ and $O$ respectively, 
	or (up to a factor $1/2$) the number of homozygous individuals $DD$ and $OO$ respectively;
    \item for each allele $a\in\{D,O\}$, $\sigma_a$ is the spatial diffusion rate of homozygous $aa$ individuals, $\omega_a$ is
	their juvenile survival rate (the proportion of newborns that survive until the adult age), $\beta_a$ is the 
	fecundity rate of a gamete carrying $a$ (fecundity being defined here multiplicatively in the sense that a mating between $aa$ and $AA$ individuals will produce $\beta_a \beta_A$ newborns), $d_a$ is the death rate of $aa$ individuals, and all these parameters
	are positive constants;
    \item $B$ and $D$, nonnegative functions of the total population size, are the per capita birth rate and death rate. 
	If they were positive constants, then a pure wild-type population would undergo exponential growth, which
	is not realistic; by using non-constant, \textit{a priori} nonlinear, functions, we can model more complex 
	population dynamics, like logistic growth or Allee effects.
\end{itemize}

Our assumptions of early and $100\%$ successful gene conversion (homing) are reflected in the presence of the $+2 \beta_{O}n_O$ term in the first equation of system~\eqref{sys:initial_full_system}; assuming late gene conversion and a dominant drive allele, together with $100\%$ gene conversion, results in the same equations.

Defining
\begin{equation*}
    \tilde{B}=\omega_O\beta_O^2 B,\ \tilde{D}=d_O D,\ \tilde{\omega}_D=\omega_D/\omega_O,\ \tilde{\beta}_D=\beta_D/\beta_O,\ 
    \tilde{d}_D=d_D/d_O,
\end{equation*}
we obtain the reduced system:
\begin{equation*}
    \begin{cases}
	\partial_t n_D -\sigma_D\Delta n_D = n_D\left( \tilde{\omega}_{D}\tilde{\beta}_{D}\tilde{B}(n_D+n_O)\frac{2n_O+\tilde{\beta}_{D}n_D}{n_D+n_O}-\tilde{d}_{D}\tilde{D}(n_D+n_O) \right), \\
	\partial_t n_O -\sigma_O\Delta n_O = n_O\left( \tilde{B}(n_D+n_O)\frac{n_O}{n_D+n_O}-\tilde{D}(n_D+n_O) \right).
    \end{cases}
\end{equation*}
Assuming subsequently that selection does not act on mobility ($\sigma_D=\sigma_O$), changing the variable $x$ into 
$\tilde{x}=x/\sqrt{\sigma_O}$, and getting rid of all ``$\tilde{\phantom{e}}$'' for ease of reading, we obtain the reduced system:
\begin{equation*}
    \begin{cases}
	\partial_t n_D -\Delta n_D = n_D\left( \omega_{D}\beta_{D}B(n_D+n_O)\frac{2n_O+\beta_{D}n_D}{n_D+n_O}-d_{D}D(n_D+n_O) \right), \\
	\partial_t n_O -\Delta n_O = n_O\left( B(n_D+n_O)\frac{n_O}{n_D+n_O}-D(n_D+n_O) \right).
    \end{cases}
\end{equation*}
The biological meaning of the parameters in this simpler system is the following: $\omega_D$, $\beta_D$, $d_D$ are multiplicative
variations to the ``norm'' fixed by $OO$ individuals. For instance, if $\beta_D=1/2$, then wild-type homozygotes ($OO$) are twice more fecund than drive homozygotes ($DD$).

Finally, when selection only acts on survival with a \textit{fitness cost} $s\in[0,1]$ ($d_D=\beta_D=1$, $\omega_D=1-s$, standing 
assumptions from now on), the system reads:
\begin{equation}
    \begin{cases}
	\partial_t n_D -\Delta n_D = n_D\left( (1-s)\left( 1+\frac{n_O}{n_D+n_O} \right)B(n_D+n_O)-D(n_D+n_O) \right), \\
	\partial_t n_O -\Delta n_O = n_O\left( \frac{n_O}{n_D+n_O}B(n_D+n_O)-D(n_D+n_O) \right).
    \end{cases}
    \label{sys:full_densities}
\end{equation}

Defining $n(t,x)=n_D(t,x)+n_O(t,x)$ the total population density and $p(t,x)=\frac{n_D(t,x)}{n(t,x)}$ the proportion of allele $D$, 
we find that the equivalent system satisfied by $(p,n)$ is:
\begin{equation}
    \begin{cases}
	\partial_t p -\Delta p -2\nabla\left( \log n \right)\cdot\nabla p=B(n)ps(1-p)\left(p-\frac{2s-1}{s}\right), \\
	\partial_t n -\Delta n=n\left( \left( 1-s+s(1-p)^2 \right)B(n)-D(n) \right).
    \end{cases}
    \label{sys:full_proportion}
\end{equation}
In the above system and in the whole paper, $\log$ is the natural logarithm.

$B(n)$ and $D(n)$ are still unspecified at this point.

The equation for $p$ in system~\eqref{sys:full_proportion} has a non-trivial equilibrium $\theta = \frac{2 s - 1}{s}$, which is admissible when $\frac{1}{2} \leq s \leq 1$. In this case, this internal equilibrium is unstable. The release-threshold properties of the gene drive are therefore given by the value of $s$: the drive is threshold-independent when $0\leq s \leq 1/2$, and threshold dependent when $1/2 < s \leq 1$, the value of the release threshold being $\theta$. 

\subsection{Carrying capacity and gene drive typology}

A pure wild-type population is governed by the equation $\partial_t n_O -\Delta n_O=n_O(B(n_O)-D(n_O))$. Hence
$B-D$ can be understood as the wild-type intrinsic growth rate per capita. 
Subsequently, we define the \textit{wild-type carrying capacity} (\textit{i.e.}, equilibrium population size) as follows:
\begin{itemize}
    \item if $B(n)-D(n)<0$ for any $n\geq 0$, then the carrying capacity is $0$;
    \item otherwise, the carrying capacity is the maximal value of $n\geq 0$ for which $B(n)-D(n)\geq 0$.
\end{itemize}
For well-posedness purposes, we assume that the carrying capacity is finite, and without loss of generality we assume that it is $1$, fixing
the unit of population density\footnote{For any $K>0$, the pair $(n_D,n_O)$ with birth and death rates $B$ and $D$ satisfies the same system as the pair $({n}^K_D,{n}^K_O)=(n_D/K,n_O/K)$ with birth and death rates ${B}^K(n)=B(Kn)$, ${D}^K(n)=D(Kn)$.
Hence the normalization can be done without loss of generality and all forthcoming results do not depend on the wild-type 
carrying capacity.}.

In a spatially homogeneous setting, a population goes extinct if, and only if, its carrying capacity is zero. The growth rate of a pure $DD$ population is $(1-s)B(n_D)-D(n_D)$. Hence, in our framework,
\begin{itemize}
    \item $s=0$ corresponds to pure, costless, replacement drives;
    \item as soon as $s>0$, the drive is a suppression drive, lowering the carrying capacity of the population;
    \item the necessary and sufficient condition for the drive to be an eradication drive, namely a drive whose carrying
	capacity is $0$, is:
\begin{equation*}
    \max_{n\in[0,1]}\left( (1-s)B(n)-D(n) \right)<0.
\end{equation*}
\end{itemize}

Finally, a simple rescaling of space and time shows that the pair $(B,D)$ and the pair $(\alpha B,\alpha D)$, for any $\alpha>0$,
lead to the same system.
In order to fix the ideas, in what follows, the pair $(B,D)$ is normalized by the standing assumption $\min_{n\in[0,1]} D(n) =1$.

\subsection{Relations with other models from the literature}

In the recent literature, reaction--diffusion models related to system~\eqref{sys:full_proportion_nonoverlapping_gen} have attracted a special attention, be it for the study of gene drive or in other evolution or population genetics contexts \footnote{The shared history between reaction--diffusion PDEs and population genetics is ancient: we remind the reader that both Fisher \cite{Fisher_1937} and Kolmogorov, Petrovsky and Piskunov \cite{KPP_1937} introduced the equation that is now famously called the Fisher--KPP equation as a population genetics model.}. Let us show how our model relates to some of those earlier models.

Fixing in the following discussion $D(n)=1$, system~\eqref{sys:full_proportion_nonoverlapping_gen} becomes:
\begin{equation}
    \begin{cases}
	\partial_t p -\Delta p -2\nabla\left( \log n \right)\cdot\nabla p=B(n)ps(1-p)\left(p-\frac{2s-1}{s}\right), \\
	\partial_t n -\Delta n=n\left( \left( 1-s+s(1-p)^2 \right)B(n)-1 \right).
    \end{cases}
    \label{sys:full_proportion_nonoverlapping_gen}
\end{equation}

\subsubsection{Deriving ref. \cite{Nadin_Strugarek_Vauchelet, Girardin_Calvez_Debarre, Tanaka_Stone_N} from system~\eqref{sys:full_proportion_nonoverlapping_gen}\label{sec:other_models}}

Recall that in this discussion, $D(n)=1$.

If $B(n)$ is a positive constant (population subjected to pure, unrealistic, Malthusian growth), then up to a rescaling of $(t,x)$, 
the equation on $p$ reads
\begin{equation}
    \partial_t p -\Delta p -2\nabla\left( \log n \right)\cdot\nabla p = sp(1-p)\left(p-\frac{2s-1}{s}\right).
    \label{eq:no_dens_dep}
\end{equation}
Together with V. Calvez, we studied this equation in a gene drive context \cite[Section 4.3]{Girardin_Calvez_Debarre} without 
\textit{a priori} knowledge on $n$. It has also been studied in a mosquito--Wolbachia context by Nadin, Strugarek and Vauchelet 
in \cite{Nadin_Strugarek_Vauchelet}.

Back to \eqref{sys:full_proportion_nonoverlapping_gen} with a non-constant function $B$, recalling that the wild-type carrying capacity
is $1$ (namely, $B(1)=1$), we consider small variations about $n=1$ (a replacement drive) by defining a ``rescaled total population'' 
$n_\varepsilon = \frac{1}{\varepsilon}(1-n)$ (\textit{i.e.}, $n=1-\varepsilon n_\varepsilon$) and, following Strugarek--Vauchelet
\cite{Strugarek_Vauchelet}, we find
\begin{equation*}
    \begin{cases}
	\partial_t p -\Delta p +2\varepsilon\frac{\nabla n_\varepsilon}{1-\varepsilon n_\varepsilon}\cdot\nabla p=B(1-\varepsilon n_\varepsilon)ps(1-p)\left(p-\frac{2s-1}{s}\right), \\
	\varepsilon\left( \partial_t n_\varepsilon -\Delta n_\varepsilon \right)=-\left( 1-\varepsilon n_\varepsilon \right)\left( \left( 1-s+s(1-p)^2 \right)B(1-\varepsilon n_\varepsilon)-1 \right).
    \end{cases}
\end{equation*}

Formally, if $\varepsilon\to 0$ with the scalings $1-\varepsilon n_\varepsilon\not\to 0$, $\varepsilon \partial_t n_\varepsilon\to 0$,
$\varepsilon\nabla n_\varepsilon\to 0$, $\varepsilon \Delta n_\varepsilon\to 0$, this system becomes
\begin{equation*}
    \begin{cases}
	\displaystyle \lim_{\varepsilon\to 0}B(1-\varepsilon n_\varepsilon)=\frac{1}{1-s+s(1-p)^2},\medskip\\
	\displaystyle \partial_t p -\Delta p = \frac{sp(1-p)\left(p-\frac{2s-1}{s}\right)}{1-s+s(1-p)^2}.
    \end{cases}
\end{equation*}
This limit can actually be made rigorous \cite{Strugarek_Vauchelet} provided:
\begin{itemize}
    \item the birth rate $B=B_\varepsilon$ depends on the small parameter $\varepsilon$ in such a way that 
    $B_\varepsilon(1-\varepsilon n_\varepsilon)$ can be rewritten as a function of $n_\varepsilon$ only, \textit{i.e.}, 
    $B_\varepsilon(1-\varepsilon n_\varepsilon)=\tilde{B}(n_\varepsilon)$;
    \item the time horizon is finite, or in other words, we are only concerned with the early gene drive dynamics, and not with its 
    long-time asymptotic spreading.
\end{itemize}

It turns out that the latter equation on $p$ is exactly the equation studied in a gene drive spreading context by Tanaka--Stone--Nelson
\cite{Tanaka_Stone_N}:
    \begin{equation}
	\partial_t p -\Delta p = \frac{sp(1-p)\left(p-\frac{2s-1}{s}\right)}{1-s+s(1-p)^2}.
	\label{eq:TSN}
    \end{equation}
Although Tanaka--Stone--Nelson's paper \cite{Tanaka_Stone_N} and Strugarek--Vauchelet's paper \cite{Strugarek_Vauchelet} were published
the same year and without direct knowledge one from another, they can therefore be related \textit{a posteriori}: our general
model reduces to Tanaka--Stone--Nelson's model via Strugarek--Vauchelet's limiting procedure in a context of replacement drives invading 
populations at carrying capacity, and Tanaka--Stone--Nelson's model can be used to study spreading properties provided the 
Strugarek--Vauchelet limiting procedure remains valid with an infinite time horizon. 
In this regard, our paper can be understood as an investigation of the validity of the approximation \eqref{eq:TSN} to study spreading
properties and for suppression or eradication drives.

Back to system~\eqref{sys:full_proportion_nonoverlapping_gen}, it is formally clear that, under a weak selection assumption $s\simeq 0$,
the total population $n$ does not depend a lot on the allelic frequency $p$, is close to the wild-type carrying capacity $1$ 
everywhere, and by virtue of $B(1)=1$, the equation on $p$ can be approximated by
\begin{equation}
    \partial_t p -\Delta p = sp(1-p)\left(p-\frac{2s-1}{s}\right).
    \label{eq:no_dens_dep_at_all}
\end{equation}
This equation is both an approximation of \eqref{eq:no_dens_dep} with $n\simeq 1$ and an approximation of \eqref{eq:TSN}
with $s\simeq 0$. This formal reasoning illustrates in another way how our model is consistent with the model of 
Tanaka--Stone--Nelson \cite{Tanaka_Stone_N} under a weak selection assumption \footnote{It was precisely the aim of Strugarek and 
Vauchelet \cite{Strugarek_Vauchelet} to make such formal statements rigorous.}.

Finally, another variant might be
\begin{equation}
    \partial_t p -\Delta p -2\nabla\left( \log n \right)\cdot\nabla p = \frac{sp(1-p)\left(p-\frac{2s-1}{s}\right)}{1-s+s(1-p)^2}.
    \label{eq:GCD}
\end{equation}
The difference with eq.~\eqref{eq:TSN} is the presence of the advection term $2\nabla\left( \log n \right)\cdot\nabla p$. 
In our paper with V. Calvez \cite{Girardin_Calvez_Debarre}, we also studied equation~\eqref{eq:GCD}, without \textit{a priori}
knowledge on $n$. It is not clear how this equation can be deduced from our current model and we leave this as an open problem -- yet,
in the next section we will explain how all these models descend from a unique discrete-time model.
In any case, our intent there was to
depart from the Tanaka--Stone--Nelson model and to introduce basic population density effects by means of this advection term. 
As such, this first work motivated the present one and this modeling clarification as well.

\subsubsection{Deriving all models from a unique discrete-time one}
Alternatively, we can derive all these models from a discrete-time model with non-overlapping generations \cite{Vella_2017, Deredec_2008, Unckless_2015}.

To obtain \eqref{sys:full_proportion_nonoverlapping_gen}, the operations have to be performed in the following order: 
\begin{enumerate}
    \item write the equations for $n_D(t+1)$ and $n_O(t+1)$;
    \item subtract from it $n_D(t)$ and $n_O(t)$ respectively and perform a first-order Taylor expansion $n_{D,O}(t+1)-n_{D,O}(t)\simeq
	\partial_t n_{D,O}$;
    \item add spatial diffusion to the equations;
    \item then write the equivalent system satisfied by $(p,n)$; it is \eqref{sys:full_proportion_nonoverlapping_gen} indeed.
\end{enumerate}

Going back to refs.~\cite{Tanaka_Stone_N,Girardin_Calvez_Debarre}, 
we find that there the operations are performed in the following different order: 
\begin{enumerate}
    \item write the equations for $n_D(t+1)$ and $n_O(t+1)$;
    \item write the equivalent system satisfied by $(p(t+1),n(t+1))$;
    \item by Taylor expansion, write the continuous-time system satisfied by $(p,n)$;
    \item add spatial diffusion (with \cite{Girardin_Calvez_Debarre} or without \cite{Tanaka_Stone_N} the gene flow advection term).
\end{enumerate}
Although the obtained equation on $n$ remains the same as in \eqref{sys:full_proportion_nonoverlapping_gen}, 
the equation on $p$ differs and is \eqref{eq:TSN} or \eqref{eq:GCD} (depending on whether the advection term
was accounted for).

Let us highlight the fact that the two processes above lead to different systems and are not equivalent. 
Nevertheless, we pointed out in \cite[Section 4.3]{Girardin_Calvez_Debarre} that the two systems have similar qualitative behaviors.

\subsection{Main results}

In what follows, we choose to focus, for the sake of exposition, on the simplest non-constant choice for $B$ and $D$, 
namely a constant death rate with logistic wild-type growth: $D(n)=1$, $B(n)=r(1-n)+1$. The parameter $r>0$ is the 
intrinsic growth rate of a wild-type population, or \textit{wild-type intrinsic growth rate} for short 
(in the literature,
intrinsic growth rates are also sometimes known as Malthusian growth rates). 
Thus we have two parameters in the model, the fitness cost $s$ and the wild-type intrinsic growth rate $r$.
We will discuss more complicated choices in Section \ref{sec:discussion}.

The reaction--diffusion systems are:
\begin{equation}\label{sys:dens_logistic}
    \begin{cases}
	\partial_t n_D -\Delta n_D = n_D\left( (1-s)\left( 1+\frac{n_O}{n_D+n_O} \right)(r(1-n_D-n_O)+1)-1 \right), \\
	\partial_t n_O -\Delta n_O = n_O\left( \frac{n_O}{n_D+n_O}(r(1-n_D-n_O)+1)-1 \right),
    \end{cases}
\end{equation}
\begin{equation}\label{sys:freq_logistic}
    \begin{cases}
	\partial_t p -\Delta p -2\nabla\left( \log n \right)\cdot\nabla p=(r(1-n)+1)sp(1-p)\left(p-\frac{2s-1}{s}\right), \\
	\partial_t n -\Delta n=n\left( \left( 1-s+s(1-p)^2 \right)(r(1-n)+1)-1 \right).
    \end{cases}
\end{equation}

For these systems, the eradication condition is $r<\frac{s}{1-s}$.

Since we are concerned with spreading properties, and more precisely with the ability of the gene drive to invade a wild-type 
population when introduced in sufficiently high numbers in a confined area of space,
we neglect boundary effects, assume that the physical space is a Euclidean space
-- to simplify even further, we assume it is the
one-dimensional Euclidean space $\mathbb{R}$ -- and assume that the initial conditions are close to $n(0,x)=1$ everywhere
and $p(0,x)=1$ in a sufficiently large compact interval.

\subsubsection{Numerical results}

All numerical simulations are performed in \textit{GNU Octave} \cite{Octave} using standard implicit finite difference schemes. 
Two examples of numerical code, that can be used to obtain Figure \ref{fig:time_evolution} and Figure \ref{fig:heatmap_large_r_++},
are presented online at \url{\urlrepo}.

To begin with, we simulate the evolution in time of the population densities $n_O(t,x)$ and $n_D(t,x)$.

The spatial domain $\mathbb{R}$ is, as usual, approximated by a very large bounded domain with Neumann boundary conditions.
Our initial condition is such that $(n_D,n_O)=(0,1)$ on the left and $(n_D,n_O)=(0.1,0.9)$ on the right.

Time snapshots of such a simulation with $r=10/9$ and $s=1/2$ are presented in Figure~\ref{fig:time_evolution}. 
We clearly observe the rapid convergence of the solution to a \textit{traveling wave}, namely to a solution 
$(n_D,n_O)(t,x)=(N_D,N_O)(x-ct)$ with constant profile and constant wave speed $c<0$. 
From now on, we use capital letters $N_D$, $N_O$, $N$ and $P$ for the wave profiles, functions of $x-ct$, 
and lower-case letters $n_D$, $n_O$, $n$ and $p$ for the population densities, functions of $(t,x)$.

Note that by symmetry and isotropy, we would obtain an axis-symmetric figure if the gene drive was introduced on  the left. If it was introduced in the center of a twice as large interval, we would observe the gluing of two traveling waves, one spreading towards the left at speed $c<0$ and the other spreading towards the right at speed $-c>0$.

With this choice of parameters $r$ and $s$, we observe the so-called \textit{hair-trigger effect}, namely the fact that no matter how small the height and the width of the initial $n_D$, invasion occurs. This is a typical property of monostable dynamics, that fails with 
bistable dynamics. These lead on the contrary to what is referred to as \textit{threshold properties}. 

\begin{figure}
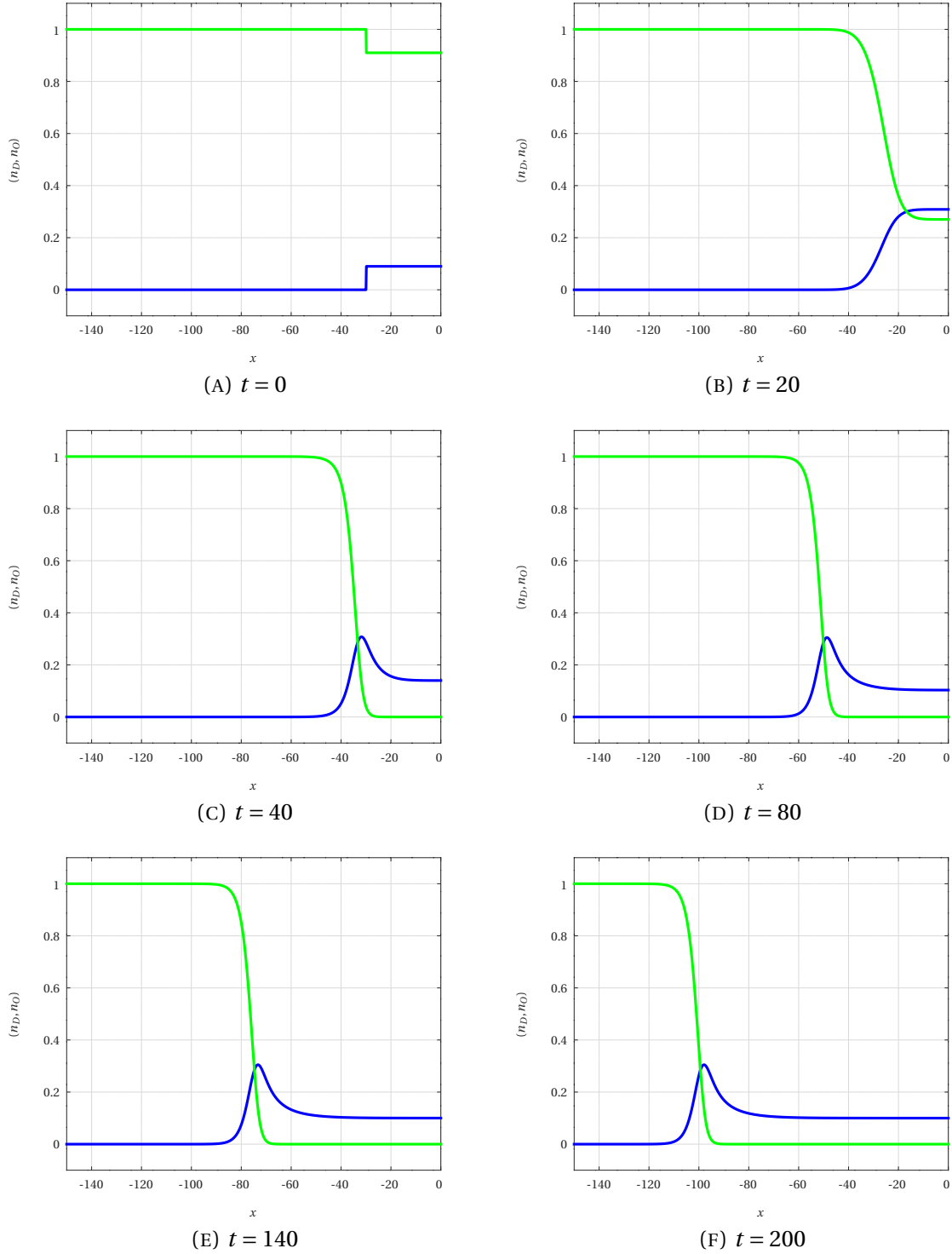

    \begin{subfigure}{.49\textwidth}
        \resizebox{\linewidth}{!}{\input{evolution_dens_0.tex}}
        \caption{$t=0$}
    \end{subfigure}
    \hfill
    \begin{subfigure}{.49\textwidth}
        \resizebox{\linewidth}{!}{\input{evolution_dens_1.tex}}
        \caption{$t=20$}
    \end{subfigure}
    
    \begin{subfigure}{.49\textwidth}
        \resizebox{\linewidth}{!}{\input{evolution_dens_2.tex}}
        \caption{$t=40$}
    \end{subfigure}
    \hfill
    \begin{subfigure}{.49\textwidth}
        \resizebox{\linewidth}{!}{\input{evolution_dens_4.tex}}
        \caption{$t=80$}
    \end{subfigure}
    
    \begin{subfigure}{.49\textwidth}
        \resizebox{\linewidth}{!}{\input{evolution_dens_7.tex}}
        \caption{$t=140$}
    \end{subfigure}
    \hfill
    \begin{subfigure}{.49\textwidth}
        \resizebox{\linewidth}{!}{\input{evolution_dens_10.tex}}
        \caption{$t=200$}
    \end{subfigure}
    \caption{Numerical simulation of the solution of the system \eqref{sys:dens_logistic} at different times (varying time between two snapshots). 
    Blue curve: $n_{D}(t,x)$. Green curve: $n_{O}(t,x)$. Here $r=10/9$, $s=0.5$, so that the drive is not
    an eradication drive ($1-s/(r(1-s))=0.1$) and it is threshold-independent ($(2s-1)/s=0$).}
    \label{fig:time_evolution}
\end{figure}

\begin{figure}
    \resizebox{\textwidth}{!}{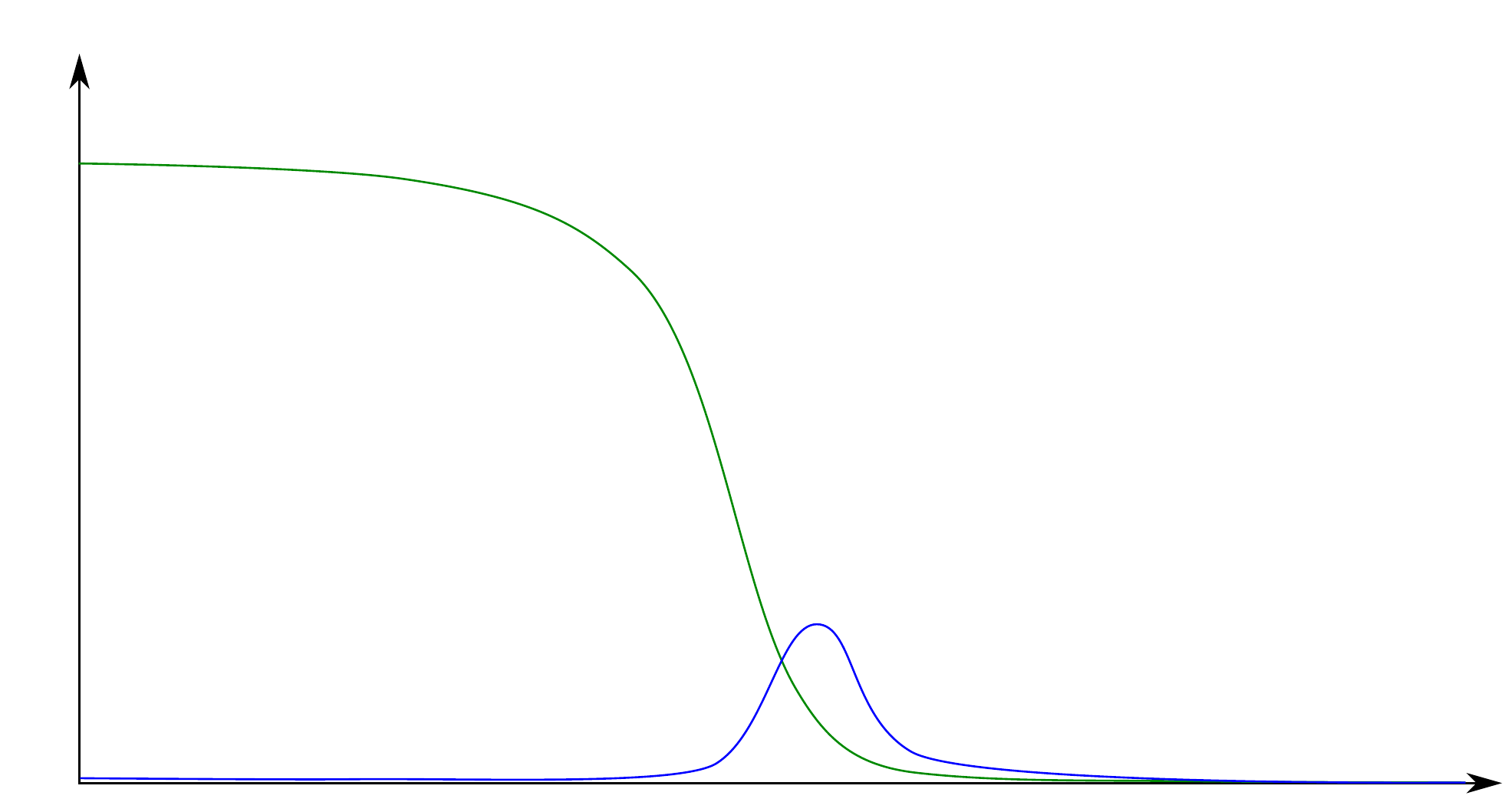}
    \caption{Illustration of a possible wave profile in the correct moving frame for an eradication drive. Three biological 
    mechanisms can influence the wave speed; finding the sign of the wave speed is studying their interplay.}
    \label{fig:wave_profile_illu}
\end{figure}

Next, we let the parameters $r$ and $s$ vary. By doing so, we observe both a monostable regime, with hair-trigger effect, 
and a bistable regime, with threshold properties.
In order to discard from our observations rapid extinctions of bistable gene drives due 
to an insufficiently large initial value of $n_D$, in all following simulations we use as initial value on the right
$(n_D,n_O)=(0.95,0.05)$ in a wide enough interval. With such a choice of initial data, we always observe on the right 
the convergence of $(n_D,n_O)$ to the steady state without wild-type $(n_D,n_O)=\left( \max\left(0,1-\frac{s}{r(1-s)}\right),0 \right)$.

We say that the gene drive invasion is successful, or that the gene drive is \textit{viable}, if the invasion speed $c_{s,r}$
is negative. On the contrary, if $c_{s,r}\geq 0$, we say that the gene drive invasion fails and that the gene drive is 
\textit{nonviable}:
in particular, when $c_{s,r}>0$, the gene drive is repelled and eliminated by the wild-type population. When introduced in the center 
of the spatial domain, a nonviable gene drive collapses in finite time.

In order to estimate the speed, it is convenient to have a well-defined level-set to follow. Hence we consider now the monotonicity
of the traveling wave profiles.
Although $N_D$ is in many cases not monotonic (see again Figure \ref{fig:time_evolution} and Figure \ref{fig:wave_profile_illu}), 
we observe in all simulated time evolutions the following monotonicities:
\begin{itemize}
    \item the wild-type population profile $N_O$ is strictly monotonic and connects $1$ to $0$;
    \item the frequency profile $P=N_D/(N_D+N_O)$ (plots of $p(t,x)$ not shown here) is monotonic -- but not always strictly, as we 
    sometimes observe $P=0$ everywhere;
    \item the total population profile $N=N_D+N_O$ (plots of $n(t,x)$ not shown here) is strictly monotonic and connects $1$ 
    to $\max\left(0,1-\frac{s}{r(1-s)}\right)$.
\end{itemize}
We systematically tested the above monotonicity properties of $P$ and $N$ in a wide parameter range for $s$ and $r$. 
The result is displayed on Figure~\ref{fig:monotonicity_heatmap} and confirms the preceding claim 
-- note that $N_O'=((1-P)N)'=-P'N+(1-P)N'$ is negative as soon as $-P'$ and $N'$ are respectively
nonpositive and negative, whence it suffices to check the monotonicities of $P$ and $N$.
Therefore we can numerically estimate the speed $c_{s,r}$ by tracking the well-defined $1/2$-level set of $n_O(t,x)$.

\begin{figure}
    \resizebox{.6\textwidth}{!}{\input{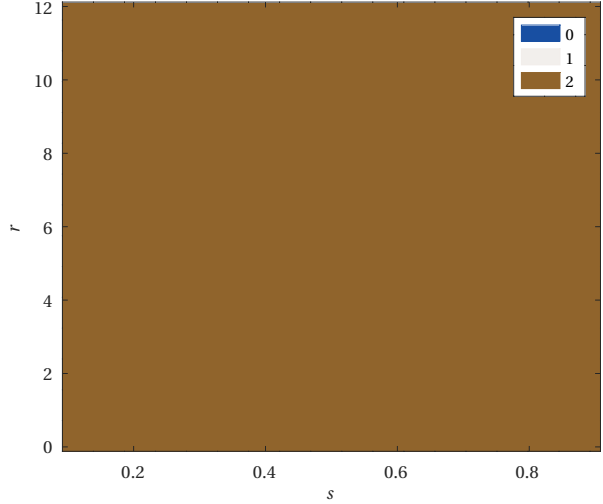}}
    \caption{Monotonicity heatmap in the $(s,r)$ plane. The output is an integer in $\{0,1,2\}$ that approximates
    the number of monotonic functions in $\{P,N\}$. The monotonicity is tested by testing the inequalities
    $\min_x\partial_x p(T,x)/\|\partial_x p(T,x)\|_{L^\infty}>-\varepsilon$ and 
    $\max_x\partial_x n(T,x)/\|\partial_x n(T,x)\|_{L^\infty}<\varepsilon$, with $T=200$ the final time of the simulation, when
    the traveling wave regime is reached. With no 
    margin of error ($\varepsilon=0$), the tests systematically fail due to numerical artifacts. Here, $\varepsilon=10^{-6}$ and the whole
    parameter range is filled in brown color: at each point $(s,r)$, the normalized derivatives of $x\mapsto p(T,x)$ and $x\mapsto n(T,x)$
    have both a constant sign up to a margin of error of $10^{-6}$.}
    \label{fig:monotonicity_heatmap}
\end{figure}

Estimations of the values of $c_{s,r}$ when $s$ and $r$ vary are displayed on heatmaps on Figure \ref{fig:heatmaps}, where
they are also compared to similar estimations for the density-independent equations \eqref{eq:no_dens_dep_at_all} and
\eqref{eq:TSN} (both equations being in some way related to our model, cf. Section \ref{sec:other_models}).
Let us note that for the equation \eqref{eq:no_dens_dep_at_all}, whose reaction term is a classical cubic nonlinearity,
an explicit formula is known for the bistable wave speed: $c_{s,r}=c_s=(2-3s)/\sqrt{2s}$. Hence its $0$-level set is exactly
at $s=2/3$. For the more complicated equation \eqref{eq:TSN} with a denominator,
we do not know any algebraic formula for the wave speed but the 
$0$-level set can still be computed approximately \cite{Tanaka_Stone_N}: $s\simeq 0.697$.
Let us also point out that despite the fact that equation \eqref{eq:no_dens_dep_at_all} arises as a weak selection 
approximation in Section \ref{sec:other_models}, we plot its values for relatively large values of $s$: in our opinion, it is interesting 
to compare on one hand the equation \eqref{eq:TSN} and its weak selection approximation \eqref{eq:no_dens_dep_at_all}, as was done
in Tanaka--Stone--Nelson \cite{Tanaka_Stone_N}, and on the other hand the equation \eqref{eq:TSN} and our system \eqref{sys:dens_logistic}. 
In particular, such a comparison clearly shows that accounting for population dynamics has a much stronger effect than accounting for 
strong selection.

The following comments on Figure~\ref{fig:heatmap_large_r} are in order.
\begin{itemize}
    \item The wave speed $c_{s,r}$ is (surprisingly to us) not a continuous function of $(s,r)$. More precisely, a jump discontinuity divides into two parts the parameter region corresponding to nonviable eradication drives. Each part corresponds to a different kind of
    nonviable eradication drive: on the left-hand side of the discontinuity, we observe drives with $P$ strictly increasing (case $c>0$ of Figure~\ref{fig:wave_profile_illu}), whereas on the right-hand side we observe drives with $P=0$ identically. 
    For the latter kind, the traveling wave observed numerically (see Figure \ref{fig:time_evolution_to_KPP})
	is actually a classical Fisher--KPP traveling wave for the $n_O$ population: after having been eradicated from the right-hand
	side of the domain by the gene drive and having waited long enough for the eradication of the gene drive itself, 
	the wild-type population invades the new open space at its Fisher--KPP speed $2\sqrt{r}$. This is confirmed graphically by the 
	fact that in the bottom right corner, the speed $c_{s,r}$ only depends on $r$. 
    \item Away from this discontinuity, the $0$-level set of $c_{s,r}$ seems to be an increasing, 
	strictly convex curve, originating from $(s,r)=(1/2,0)$ and admitting a vertical straight line as asymptote. The equation
	of this asymptote has the form $s=s_0$ for some undetermined $s_0\in[0.65,0.7]$. 
	In particular, the level set is included in $\{(s,r)\ |\ s\in[0.5,0.7]\}$: 
	if $s< 0.5$, then $c_{s,r}<0$ and the gene drive invades; if $s>0.7$, then $c_{s,r}>0$ and the gene drive is nonviable.
	These two conclusions for the parameter range $s\notin[0.5,0.7]$ were already valid for the density-independent equations
	\eqref{eq:TSN} and \eqref{eq:no_dens_dep_at_all} (see Figures \ref{fig:heatmap_no_dens_dep_at_all} and \ref{fig:heatmap_TSN}).
    \item There exist viable eradication drives, yet the corresponding parameter region is small, especially when intersected with
    the half-space defined by $s\geq 1/2$.
\end{itemize}

\begin{figure}
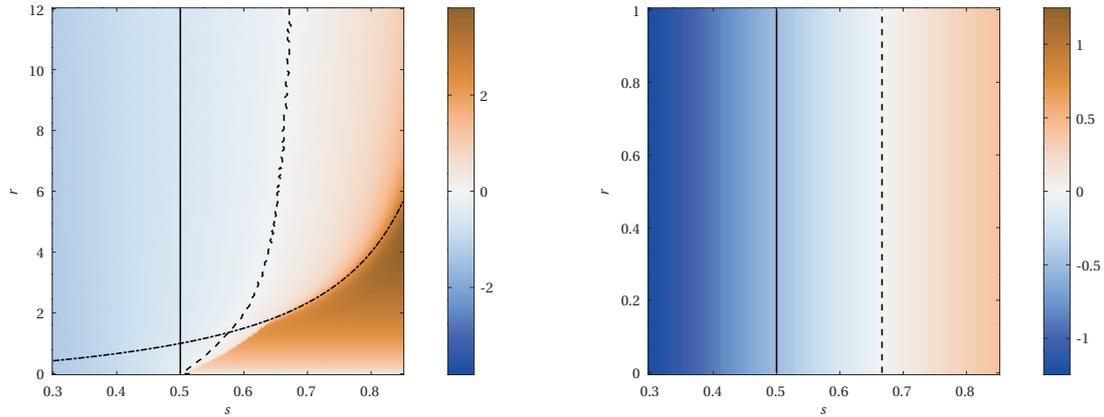
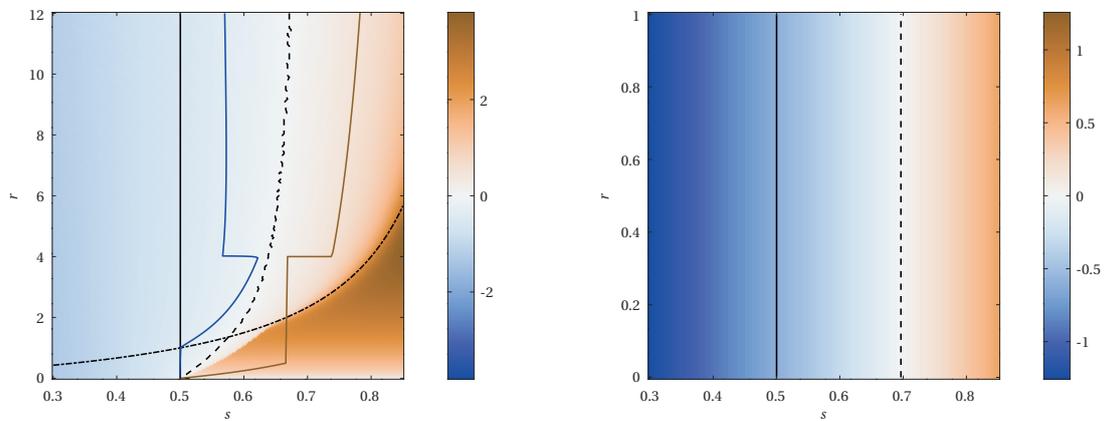

\begin{subfigure}[t]{.49\textwidth}
    \resizebox{\textwidth}{!}{\input{heatmap_large_r.tex}}
    \caption{For the system \eqref{sys:dens_logistic}.}
    \label{fig:heatmap_large_r}
\end{subfigure}
\hfill
\begin{subfigure}[t]{.49\textwidth}
    \resizebox{\textwidth}{!}{\input{heatmap_no_dens_dep.tex}}
    \caption{For the equation \eqref{eq:no_dens_dep_at_all}.}
    \label{fig:heatmap_no_dens_dep_at_all}
\end{subfigure}

\begin{subfigure}[t]{.49\textwidth}
    \resizebox{\textwidth}{!}{\input{heatmap_large_r_++.tex}}
    \caption{Same as \ref{fig:heatmap_large_r} with two colored lines delimiting a central zone where the sign is analytically unknown (cf. Theorems 
    \ref{thm:main_nonexistence} and \ref{thm:main_sign_when_nontrivial_existence}).}
    \label{fig:heatmap_large_r_++}
\end{subfigure}
\hfill
\begin{subfigure}[t]{.49\textwidth}
    \resizebox{\textwidth}{!}{\input{heatmap_TSN_no_dens_dep.tex}}
    \caption{For the equation \eqref{eq:TSN}.}
    \label{fig:heatmap_TSN}
\end{subfigure}
    \caption{Heatmaps of $c_{s,r}$ values in the $(s,r)$ plane comparing \eqref{sys:dens_logistic} and two density-independent models
    (note that the scaling of the colorbar varies). 
    Dashed curves: $0$-level set. Solid black curves: monostability--bistability threshold. 
    Dashed-dotted curve: level of $r$ below which the drive is an eradication drive.
    As the equations \eqref{eq:TSN} and \eqref{eq:no_dens_dep_at_all} do not depend on $r$, the level lines on their respective
    figures are vertical.}
    \label{fig:heatmaps}
\end{figure}

\begin{figure}
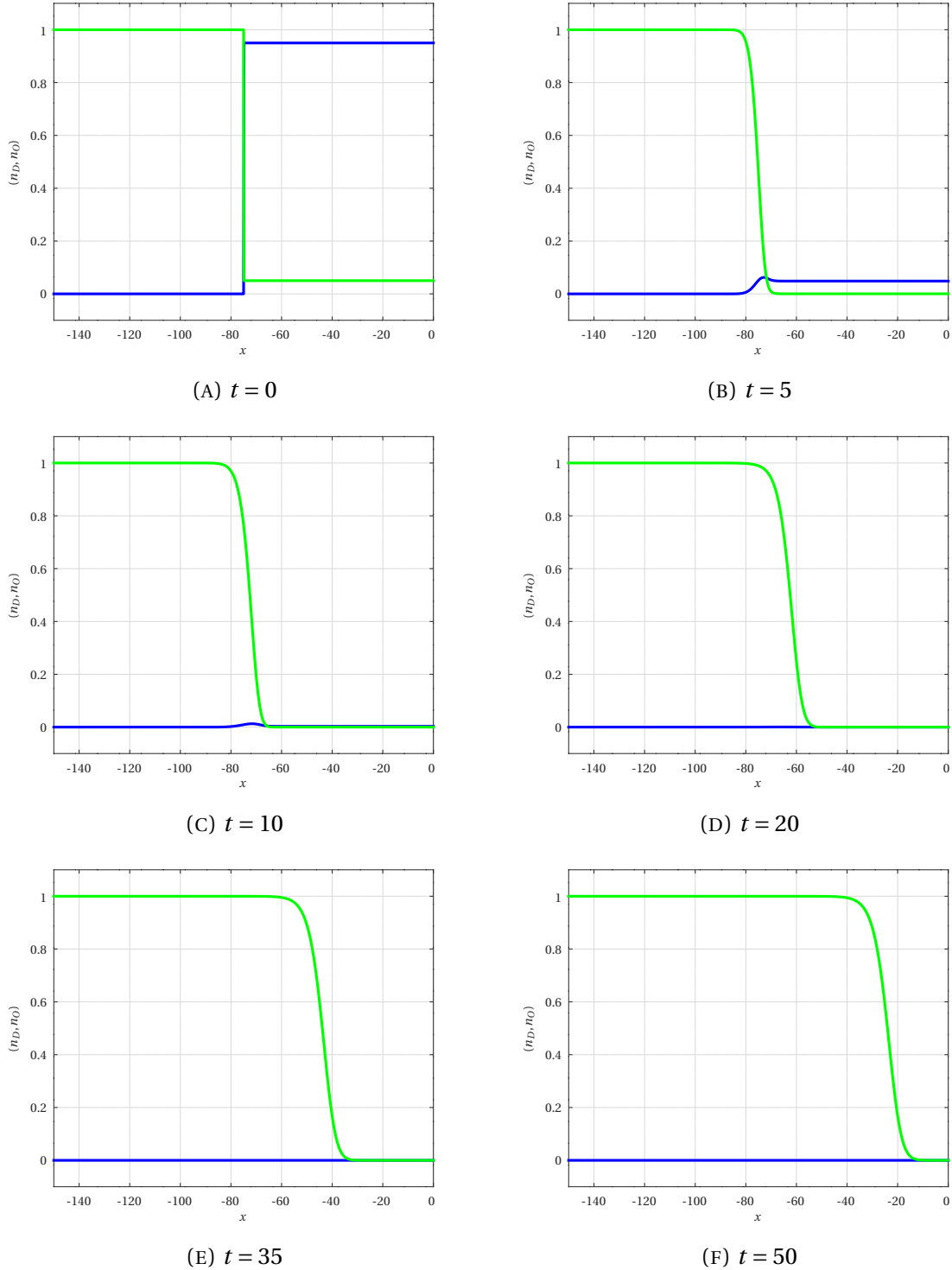

    \begin{subfigure}{.49\textwidth}
        \resizebox{\linewidth}{!}{\input{evolution_to_KPP_0.tex}}
        \caption{$t=0$}
    \end{subfigure}
    \hfill
    \begin{subfigure}{.49\textwidth}
        \resizebox{\linewidth}{!}{\input{evolution_to_KPP_1.tex}}
        \caption{$t=5$}
    \end{subfigure}
    
    \begin{subfigure}{.49\textwidth}
        \resizebox{\linewidth}{!}{\input{evolution_to_KPP_2.tex}}
        \caption{$t=10$}
    \end{subfigure}
    \hfill
    \begin{subfigure}{.49\textwidth}
        \resizebox{\linewidth}{!}{\input{evolution_to_KPP_4.tex}}
        \caption{$t=20$}
    \end{subfigure}
    
    \begin{subfigure}{.49\textwidth}
        \resizebox{\linewidth}{!}{\input{evolution_to_KPP_7.tex}}
        \caption{$t=35$}
    \end{subfigure}
    \hfill
    \begin{subfigure}{.49\textwidth}
        \resizebox{\linewidth}{!}{\input{evolution_to_KPP_10.tex}}
        \caption{$t=50$}
    \end{subfigure}
    \caption{Numerical simulation of the solution of the system \eqref{sys:dens_logistic} at different times (varying time between two snapshots). 
    Blue curve: $n_{D}(t,x)$. Green curve: $n_{O}(t,x)$. Here $s=0.7$, $r=0.5$,
    so that by Theorem \ref{thm:main_nonexistence}, traveling waves are trivial; moreover,
    $(2s-1)/s=4/7<0.95/(0.95+0.05)=0.95$, so that the extinction of $n_D$ is not due to the bistability threshold.}
    \label{fig:time_evolution_to_KPP}
\end{figure}

\subsubsection{Analytical results}\label{sec:analytical_results}
Although finding the sign of the wave speed for a scalar bistable reaction--diffusion equation is easy (multiply the equation
by the derivative of the wave profile and integrate over $\mathbb{R}$), it is in general a very challenging problem
for bistable reaction--diffusion systems of equations. The method used in the scalar case only works for systems whose reaction term
has a specific gradient form -- which is not the case here -- and, apart from this method, no general method is known.
Systems devoid of gradient form have to be studied on a case-by-case basis and most of the time these studies lead to
results on very particular cases with specific algebraic requirements on the parameters of the system.
For more details on this topic, we refer to the recent review by the first author on the sign of 
the wave speed for two-species Lotka--Volterra competition--diffusion systems \cite{Girardin_2019}.

Hence we believe that, given the present state of knowledge in the mathematical analysis of reaction--diffusion systems, an explicit complete characterization of the sign of $c_{s,r}$ is currently out of reach. We can nonetheless
aim for a collection of partial results. In our opinion, the value of such results is twofold:
\begin{itemize}
    \item on one hand, they confirm in some regions of the parameter space the numerical experiments;
    \item on the other hand, their proofs give precious insights on the deep structure of the equations, that 
	might be useful in possible future work.
\end{itemize}

Regarding the existence of traveling waves $(p,n)(t,x)=(P,N)(x-ct)$ with monotonic $P$ and $N$, we point out first that the 
existence in the case $P=0$ identically reduces to a standard question for the Fisher-KPP equation and the conclusion is well known: 
there exists such a traveling wave with speed $c$ if and only if $c\geq 2\sqrt{r}$ and the traveling wave with speed 
$c\geq 2\sqrt{r}$ is unique (up to spatial translation).
On the contrary, the case $P\neq 0$ is a much more delicate issue here than in standard systems.
On one hand, we will prove with Theorem \ref{thm:main_nonexistence} an explicit nonexistence result in a certain parameter range, 
showing that the existence for all values of $(r,s)$ is simply false. On the other hand, known methods for constructing traveling waves rely either
on monotonicity properties of the reaction term and super-sub-solutions or on ODE shooting arguments with stable or unstable manifolds. 
Yet, both methods seem inappropriate here:
both systems \eqref{sys:dens_logistic} and \eqref{sys:freq_logistic} have changing monotonicities and lack $\mathscr{C}^1$ 
regularity as $n\to 0$. We believe we might be able to obtain
analytical existence results in certain parameter ranges (say, away from the eradication zone and for small values of $s$ 
where the system has a nice KPP structure \cite{Girardin_2016_2}), but not in the range that interests us most, namely 
$s\in\left[ 0.5,0.7 \right]$ and in the eradication zone or close to it.
Therefore we decide to leave the rigorous existence problem as an open problem\footnote{The empirical proof of existence 
provided by our numerical experiments might be sufficiently convincing for many non-mathematicians.}.

In any case, the biologically most relevant question is not the question of existence but rather the question of the direction
of the propagation, namely the sign of the wave speed $c_{s,r}$. Hence in the forthcoming Theorem \ref{thm:main_sign_when_nontrivial_existence}
we will focus indeed on \textit{a priori} estimates for this sign, or in other words we will study the sign of the wave speed of 
any existing traveling wave.

We use the variables $(p,n)$ to write the following theorems; the analogous results with variables
$(n_D,n_O)$ can be written easily. In order to simplify the statements, we exclude traveling waves
that are not monotonic or do not converge exponentially at $-\infty$; this decision is consistent
with numerical observations. We also choose to list only the results that give explicit regions of the
$(s,r)$ plane where the sign is known, but our proofs actually give slightly larger implicit regions.

\begin{defi}\label{def:tw}
    A \textit{traveling wave solution} of \eqref{sys:freq_logistic} is a bounded nonnegative classical solution of the 
    form $(p,n)(t,x)=(P,N)(x-ct)$ satisfying:
    \begin{enumerate}
        \item $P$ is nondecreasing, $N$ is decreasing;
        \item $\lim_{-\infty}(P,N)=(0,1)$.
    \end{enumerate}
    Furthermore, the traveling wave is referred to as \textit{trivial} if $P=0$ and \textit{nontrivial} if $P\neq 0$.
\end{defi}

Obviously, a trivial, respectively nontrivial, traveling wave solution $(n_D,n_O)$ of \eqref{sys:dens_logistic} is a 
solution such that the associated solution $(p,n)$ of \eqref{sys:freq_logistic} is a trivial, respectively nontrivial,
traveling wave solution itself.

Our two main results follow. Graphically, they are summarized on Figure \ref{fig:heatmap_large_r_++}.

\begin{thm}[Nonexistence of nontrivial waves]
    \label{thm:main_nonexistence}
    If $s>\frac12$ and 
    \[
	0<r\leq\frac{2s-1}{2(1-s)},
    \]
    then all traveling waves, in the sense of Definition \ref{def:tw}, are trivial. 
    
    Consequently, for any traveling wave solution $(p,n)$, $n(t,x)=N(x-ct)$ is a Fisher--KPP traveling wave with speed $c\geq 2\sqrt{r}$.
\end{thm}

Figure \ref{fig:time_evolution_to_KPP} is an example of spreading dynamics when $s>\frac12$ and $r\leq\frac{2s-1}{2(1-s)}$.

\begin{thm}[Sign of nontrivial wave speed]
    \label{thm:main_sign_when_nontrivial_existence}
    Assume \eqref{sys:freq_logistic} admits a nontrivial traveling wave solution $(p,n)(t,x)=(P,N)(x-ct)$, in the sense
    of Definition \ref{def:tw}, such that:
    \begin{enumerate}
	\item $P$ is strictly monotonic;
	\item $P$ converges exponentially fast to $0$ at $-\infty$.
    \end{enumerate}

    Then:
    \begin{enumerate}
	\item $c<0$ if:
            \begin{enumerate}
		\item $s\leq\frac12$; moreover, $P$ converges at $+\infty$ to $1$ and $c\leq-2\sqrt{1-2s}$;
		\item $s\in\left[ \frac12,\frac{2}{3} \right]$, $r>\frac{s}{1-s}$, 
		    $(P,N)$ converges at $+\infty$ to $\left( 1,\frac{1}{r}\left( r-\frac{s}{1-s} \right) \right)$ and 
		    \begin{equation*}
			4\geq r\geq\left( \frac{s}{1-s} \right)\left( 1-\left(\frac{(1-s)\left( \frac{2s-1}{s} \right)^3}{2-3s+\left( \frac{2s-1}{s} \right)^3}\right)^{1/4} \right)^{-1}> 0;
		    \end{equation*}
		\item $s\in\left[ \frac12,\frac{2}{3} \right]$, $r\geq 4$, $(P,N)$ converges at $+\infty$ to $\left( 1,\frac{1}{r}\left( r-\frac{s}{1-s} \right) \right)$ and 
                    \begin{equation*}
                        \left( 1-\frac{s}{r(1-s)} \right)^4\left( 2-3s \right)\geq\left( r+1-\left( 1-\frac{s}{r(1-s)} \right)^4 \right)\left( \frac{2s-1}{s} \right)^3;
                    \end{equation*}
            \end{enumerate}
	\item $c>0$ if:
	    \begin{enumerate}
        	\item $s\geq\frac{2}{3}$ and $r\leq 4$;
        	\item $s>\frac{2}{3}$ and 
		    \begin{equation*}
			r<\frac{s^3(3s-2)}{(2s-1)^3-s^3(3s-2)}.			
		    \end{equation*}
	    \end{enumerate}
    \end{enumerate}
\end{thm}

\section{Proof of Theorems \ref{thm:main_nonexistence} and \ref{thm:main_sign_when_nontrivial_existence}}

In the whole section, the spatial domain is $\mathbb{R}$, so that solutions of \eqref{sys:dens_logistic} or \eqref{sys:freq_logistic} are one-dimensional.

\subsection{Proof of Theorem \ref{thm:main_nonexistence}}

\begin{prop}
    Let $(n_D,n_O)$ be a solution of \eqref{sys:dens_logistic} set in $(0,+\infty)\times\mathbb{R}$ 
    with $n_D(0,x)\geq0$ and $n_O(0,x)>0$ for all $x\in\mathbb{R}$ and $n_D(0,\bullet)$ non-zero.

    Then:
    \begin{enumerate}
	\item if $0<r\leq\frac12\frac{2s-1}{1-s}$, $n_D$ goes extinct spatially uniformly;
	\item if $r>\frac12\frac{2s-1}{1-s}$ and if $n_D(0,\bullet)$ is compactly supported,
	    either $n_D$ does not spread or it spreads at most at speed 
	    $2\sqrt{2(1-s)\left( r-\frac12\frac{2s-1}{1-s} \right)}$.
    \end{enumerate}
\end{prop}
\begin{proof}
    Using $1+n_O/(n_D+n_O)\leq 2$ and then $1-n_D-n_O\leq 1-n_D$, we find that $n_D$ satisfies
    \begin{equation*}
	\partial_t n_D -\partial_{xx} n_D\leq n_D(-1+2(1-s)(r+1)-2(1-s)rn_D).
    \end{equation*}
    By comparison principle, $n_D\leq\overline{n_D}$, where $\overline{n_D}$ is the solution of the Fisher--KPP equation
    \begin{equation*}
	\partial_t \overline{n_D} -\partial_{xx} \overline{n_D}= \overline{n_D}(-1+2(1-s)(r+1)-2(1-s)r\overline{n_D}).
    \end{equation*}
    If $0<r\leq\frac12\frac{2s-1}{1-s}$, the super-solution goes extinct uniformly, and so does $n_D$.
    Otherwise, the asymptotic speed of spreading of the super-solution, which is an upper bound for the asymptotic speed of spreading
    of $n_D$, is exactly the speed given in the statement.
\end{proof}

\begin{cor}
    If $0<r\leq\frac12\frac{2s-1}{1-s}$, then all traveling waves are trivial. 
    
    Consequently, any wave speed $c$ satisfies $c\geq 2\sqrt{r}$.
\end{cor}

\subsection{Proof of Theorem \ref{thm:main_sign_when_nontrivial_existence}}

\subsubsection{Viable gene drives}

\begin{prop}
    Assume \eqref{sys:dens_logistic} admits a nontrivial traveling wave solution $(n_D,n_O)(t,x)=(N_D,N_O)(x-ct)$ satisfying
    \begin{equation*}
	0<\liminf_{-\infty}\frac{N_D''}{N_D},\quad 0<\liminf_{-\infty}\frac{N_D'}{N_D}.
    \end{equation*}

    Then, if $s<\frac12$, necessarily $c<0$.

    Furthermore, if there exists $\lambda>0$ such that
    \begin{equation*}
	\lim_{-\infty}\frac{N_D''}{N_D}=\lambda^2,\quad \lim_{-\infty}\frac{N_D'}{N_D}=\lambda,
    \end{equation*}
    then $s\leq\frac12$ implies $c\leq -2\sqrt{1-2s}$.
\end{prop}
\begin{proof}
    Close to $-\infty$, the wave profile $N_D$ is positive and, by virtue of the assumptions,
    increasing ($N_D'>0$) and strictly convex ($N_D''>0$). Assuming by contradiction $c\geq 0$ and plugging these inequalities
    into the traveling wave equation
    \begin{equation*}
	-N_D''-cN_D'=N_D\left( -1+\left( r(1-N_D-N_O)+1 \right)\left( 1-s \right)\left( 2-\frac{N_D}{N_D+N_O} \right) \right),
    \end{equation*}
    we get
    \begin{equation*}
	\frac{1}{1-s}>\left( r(1-N_D-N_O)+1 \right)\left( 2-\frac{N_D}{N_D+N_O} \right).
    \end{equation*}
    As the right-hand side converges to $2$ when $(N_D,N_O)\to(0,1)$, we deduce $s\geq\frac12$.

    Subsequently, assuming the exponential convergence of $N_D$ with rate $\lambda>0$ and 
    passing to the limit into the equation, we discover
    \begin{equation*}
	-\lambda^2-c\lambda=2(1-s)-1=1-2s
    \end{equation*}
    which has a positive solution if and only if $c^2\geq 4(1-2s)$.
\end{proof}

\begin{prop}
    Assume \eqref{sys:freq_logistic} admits a nontrivial traveling wave solution $(p,n)(t,x)=(P,N)(x-ct)$ with increasing $P$.

    Then, if $s\leq\frac12$, necessarily $c<0$.
\end{prop}
\begin{proof}
    First, we focus on the critical case $s=\frac12$.
    
    By classification of the constant solutions of the system \eqref{sys:freq_logistic} when $s=\frac12$, 
    necessarily such a traveling wave admits $(1,0)$ or 
    $\left(1,\frac{1}{r}\left( r-\frac{s}{1-s} \right) \right)$ as limit at $+\infty$.
    In all cases, $P$ converges to $1$.

    We use the strict monotonicity of $N$ and $P$ to establish the existence of
    \[
        h:[0,1]\to[\lim_{+\infty}N,1],
    \]
    decreasing, bijective and of class $C^2$, such that $h(P)=N$. 
    Using a change of variable discovered by Nadin, Strugarek and Vauchelet \cite{Nadin_Strugarek_Vauchelet}, we deduce 
    that there exists a traveling wave $P(x-ct)$ connecting $0$ to $1$ with increasing profile
    if and only if there exists a traveling wave solution of
    \begin{equation*}
	\partial_t v -\partial_{xx} v= f\left( H^{-1}(v) \right)h^2\left( H^{-1}(v) \right),
    \end{equation*}
    where 
    \[
    f(v)=\frac12(r(1-h(v))+1)v^2(1-v)\quad\text{and}\quad H(v)=\int_0^v h^2(z)\textup{d}z,
    \]
    with increasing profile as well.
    The reaction term $g:v\mapsto f\left( H^{-1}(v) \right)h^2\left( H^{-1}(v) \right)$ is nonnegative and, by computing its derivatives at $0$, 
    it turns out that $g'(0)=f'(0)=0$ and $g''(0)=f''(0)=1>0$: this equation is degenerate monostable. 
    By standard results on degenerate monostable equations with 
    non-degenerate second derivatives (prototypical reaction term $v^2(1-v)$), there exists a traveling wave solution of speed $c$ for this
    equation if and only if $c\leq c^\star$, where $c^\star<0$.

    Next, we consider the case $s<\frac12$. It is in fact proved similarly, except this 
    time $g'(0)=f'(0)>0$, so that we do not even need to look at $g''(0)$.
\end{proof}

\begin{prop}
    Assume \eqref{sys:freq_logistic} admits a nontrivial traveling wave solution $(p,n)(t,x)=(P,N)(x-ct)$ with limit 
    $\left( 1,\frac{1}{r}\left( r-\frac{s}{1-s} \right) \right)$ at $+\infty$ and with increasing $P$.

    Then, if $r>\frac{s}{1-s}$ and if one of the following conditions holds true, necessarily $c<0$:
    \begin{enumerate}
	\item 
\begin{equation*}
    \left( 1-\frac{s}{r(1-s)} \right)^4\left( 2-3s \right)\geq\left( r+1-\left( 1-\frac{s}{r(1-s)} \right)^4 \right)\left( \frac{2s-1}{s} \right)^3;
\end{equation*}
\item 
\begin{equation*}
    4\geq r\geq\left( \frac{s}{1-s} \right)\left( 1-\left(\frac{(1-s)\left( \frac{2s-1}{s} \right)^3}{2-3s+\left( \frac{2s-1}{s} \right)^3}\right)^{1/4} \right)^{-1}> 0.
\end{equation*}
    \end{enumerate}
\end{prop}
\begin{proof}
    The case $s\leq\frac12$ is already solved by the previous proposition, hence we assume
    without loss of generality $s>\frac12$. Moreover, the necessary nonnegativity of the limit
    implies $r\geq\frac{s}{1-s}$; the case of equality is discarded by assumption, so that
    we have $r>\frac{s}{1-s}$. 

    Using again the relation $N=h(P)$ and the change of variable of Nadin--Strugarek--Vauchelet \cite{Nadin_Strugarek_Vauchelet}, 
    we find
\begin{equation*}
    \operatorname{sign}(c)=\operatorname{sign}\left( -\int_0^1 h(V)^4(r(1-h(V))+1)Vs(1-V)\left(V-\frac{2s-1}{s}\right)\textup{d}V \right).
\end{equation*}

    Let $I$ be the integral in the right-hand side, so that $\operatorname{sign}(c)=-\operatorname{sign}(I)$.
Using the strict positivity of $h$, namely 
\[
    \min_{[0,1]} h =\frac{1}{r}\left( r-\frac{s}{1-s} \right)>0,
\]
we deduce
\begin{align*}
    I &
    > \max_{\left[ 0,\theta \right]}\left(h^4(r(1-h)+1)\right)\int_0^\theta Vs(1-V)\left(V-\theta\right)\textup{d}V \\
    & \quad +\min_{\left[ \theta,1 \right]}\left(h^4(r(1-h)+1)\right)\int_\theta^1 Vs(1-V)\left(V-\theta\right)\textup{d}V \\
    & \geq (r+1)\int_0^\theta Vs(1-V)\left(V-\theta\right)\textup{d}V \\
    & \quad +\frac{1}{r^4}\left( r-\frac{s}{1-s} \right)^4\int_\theta^1 Vs(1-V)\left(V-\theta\right)\textup{d}V,
\end{align*}
where $\theta=(2s-1)/s$. Consequently, $c<0$ if
\begin{equation*}
    \left( 1-\frac{s}{r(1-s)} \right)^4\left( -\frac{1}{12s^3}+\frac{1}{2s^2}-\frac{s}{4}-\frac{1}{s}+\frac{5}{6} \right)\geq(r+1)\left( -\frac{1}{12s^3}+\frac{1}{2s^2}-\frac{1}{s}+\frac{2}{3} \right),
\end{equation*}
that is if 
\begin{equation*}
    \left( 1-\frac{s}{r(1-s)} \right)^4\left( 2-3s \right)\geq\left( r+1-\left( 1-\frac{s}{r(1-s)} \right)^4 \right)\left( \frac{2s-1}{s} \right)^3.
\end{equation*}
This inequality is difficult to solve explicitly, however we point out that it is both true
close to $s=\frac12$ and false as $r\to+\infty$.

The function 
\[
    V\mapsto h(V)^4(r(1-h(V))+1)
\]
admits as derivative 
\[
    V\mapsto h(V)^3(4(r+1)-5rh(V))h'(V),
\]
which is negative in $(0,1)$ if and only if $r\leq 4$. Hence, in the particular case $r\leq 4$, the function 
$h^4(r(1-h)+1)$ is decreasing, so that the inequality can be improved as
\begin{equation*}
    \left( 1-\frac{s}{r(1-s)} \right)^4\frac{1}{1-s}\left( -\frac{1}{12s^3}+\frac{1}{2s^2}-\frac{s}{4}-\frac{1}{s}+\frac{5}{6} \right)\geq \left( -\frac{1}{12s^3}+\frac{1}{2s^2}-\frac{1}{s}+\frac{2}{3} \right),
\end{equation*}
which reduces to
\begin{equation*}
    \left( 1-\frac{s}{r(1-s)} \right)^4(2-3s)\geq \left( 1-s-\left( 1-\frac{s}{r(1-s)} \right)^4 \right)\left( \frac{2s-1}{s} \right)^3.
\end{equation*}
This inequality can be rewritten as
\begin{equation*}
    \left( \frac{1-s}{s} \right)\left( 1-\left(\frac{(1-s)\left( \frac{2s-1}{s} \right)^3}{2-3s+\left( \frac{2s-1}{s} \right)^3}\right)^{1/4} \right)\geq\frac{1}{r}.
\end{equation*}
This ends the proof.
\end{proof}

\subsubsection{Nonviable gene drives}

Previous propositions already confirm that nonviable gene drives can be found only in the parameter range $s>\frac12$.

When $s\geq\frac12$, the monotonicity of $(P,N)$ implies the convergence at $+\infty$ to one constant solution among 
$(0,0)$, $\left( \frac{2s-1}{s}, 1-\frac{2s-1}{r(1-s)} \right)$
(note that this value of $n$ is positive if and only if $r>\frac{2s-1}{1-s}$) and
$\left( 1,\max\left( 0,1-\frac{s}{r(1-s)} \right) \right)$.

The positivity of the speed $c$ being already established in the first case (trivial traveling waves), we focus on the second 
and third case (nontrivial traveling waves).

\begin{prop}
    Let  $(p,n)(t,x)=(P,N)(x-ct)$ be a traveling wave solution of \eqref{sys:freq_logistic}.

    If one of the two following conditions holds true, the traveling wave is nontrivial and necessarily $c>0$:
    \begin{enumerate}
	\item $P$ converges at $+\infty$ to $\frac{2s-1}{s}$ and $s\geq\frac12$;
	\item $P$ converges at $+\infty$ to $1$, $r<\frac{s^3(3s-2)}{(2s-1)^3-s^3(3s-2)}$ and $s>\frac{2}{3}$;
	\item $P$ is increasing, $P$ converges at $+\infty$ to $1$, $r\leq 4$ and $s\geq\frac{2}{3}$.
    \end{enumerate}
\end{prop}
\begin{proof}
    Without loss of generality, assume $s\geq\frac12$.

    \textbf{First condition.}
    The degenerate case $s=\frac12$ is obvious ($n(t,x)$ is a standard Fisher--KPP traveling wave), whence we focus on $s>\frac12$.

    Since $\partial_x(\log n)\cdot\partial_x p=N'P'/N\leq 0$, by comparison principle, 
    $p\leq\overline{p}$ and $c\geq\overline{c}$, where $\overline{p}$ is the solution of
\begin{equation*}
    \begin{cases}
	\partial_t \overline{p} -\partial_{xx} \overline{p} =\overline{p}s(1-\overline{p})\left(\overline{p}-\frac{2s-1}{s}\right), \\
	\overline{p}(0,x)=P(x),
    \end{cases}
\end{equation*}
and $\overline{c}$ is the asymptotic speed of spreading of $\overline{p}$.
This spreading speed is larger than or equal to the associated minimal wave speed, which is positive.

    \textbf{Second condition.}
    Since $\partial_x(\log n)\cdot\partial_x p=N'P'/N\leq 0$, by comparison principle, 
    $p\leq\overline{p}$ and $c\geq\overline{c}$, where $\overline{p}$ is the solution of
\begin{equation*}
    \begin{cases}
	\partial_t \overline{p} -\partial_{xx} \overline{p} =(r+1)\mathbf{1}_{\overline{p}\geq(2s-1)/s}\overline{p}s(1-\overline{p})\left(\overline{p}-\frac{2s-1}{s}\right)+\mathbf{1}_{\overline{p}<(2s-1)/s}\overline{p}s(1-\overline{p})\left(\overline{p}-\frac{2s-1}{s}\right), \\
	\overline{p}(0,x)=P(x),
    \end{cases}
\end{equation*}
and $\overline{c}$ is the asymptotic speed of spreading of $\overline{p}$.

The spreading speed of such a bistable equation is exactly its unique traveling wave speed
associated with an increasing profile having limits $0$ at $-\infty$ and $1$ at $+\infty$. 

Therefore the speed $\overline{c}$ has the sign of the opposite of the integral over $[0,1]$
of 
\[
V\mapsto (r+1)\mathbf{1}_{V\geq(2s-1)/s}Vs(1-V)\left(V-\frac{2s-1}{s}\right)+\mathbf{1}_{V<(2s-1)/s}Vs(1-V)\left(V-\frac{2s-1}{s}\right).
\]
After some algebra we discover that
\begin{equation*}
    \operatorname{sign}(\overline{c})=\operatorname{sign}\left((r+1)\left( s-\frac{2}{3} \right)-\frac{r}{3}\left( \frac{2s-1}{s} \right)^3\right).
\end{equation*}
Hence $\overline{c}>0$ if (and only if) 
\[
r<\frac{s^3(3s-2)}{(2s-1)^3-s^3(3s-2)}.
\]
If $s\in\left( \frac12,\frac{2}{3} \right]$, the positivity of $r$ yields a contradiction, 
so that $s>\frac{2}{3}$ is required additionally without loss of generality.
Since $p\leq\overline{p}$ implies
$c\geq\overline{c}$, this ends the proof for the second condition.

\textbf{Third condition.}
We use again the $C^2$-diffeomorphism $h$ such that $N=h(P)$ and the Nadin--Strugarek--Vauchelet formula
\cite{Nadin_Strugarek_Vauchelet}:
\begin{equation*}
    \operatorname{sign}(c)=\operatorname{sign}\left( -\int_0^1 h(V)^4(r(1-h(V))+1)Vs(1-V)\left(V-\frac{2s-1}{s}\right)\textup{d}V \right).
\end{equation*}

Recall that $V\mapsto h(V)^4(r(1-h(V))+1)$ is decreasing in $[0,1]$ if $r\leq 4$.
In such a case, it can be verified that the integral $I$ in the right-hand side above satisfies
\begin{equation*}
    I< h\left( \theta \right)^4(r(1-h(\theta))+1)\int_0^1 Vs(1-V)\left( V-\theta \right)\textup{d}V,
\end{equation*}
where $\theta=(2s-1)/s$.
Therefore, $c$ is positive if the latter integral is nonpositive, namely if $s\geq 2/3$.
\end{proof}

\section{Discussion\label{sec:discussion}}

\subsection{Main conclusions}

In \cite{Tanaka_Stone_N}, Tanaka \textit{et al.} suggested the following terminology:
\begin{itemize}
    \item monostable gene drives, susceptible of hair-trigger effect, are \textit{socially irresponsible gene drives}: the escape
	of just one individual carrying the gene drive allele from the laboratory suffices to trigger an invasion -- they correspond to threshold-independent drives;
    \item bistable gene drives, which can never invade if released accidentally in small quantities, are on the contrary \textit{socially responsible gene drives} -- they correspond to high-threshold drives.
\end{itemize}
This terminology can be combined with our viable/nonviable terminology: all socially irresponsible gene drives are viable
but socially responsible gene drives can, \textit{stricto sensu}, be viable or nonviable.
Of course, practical gene drives should all be both viable and socially responsible. In Tanaka--Stone--Nelson's density-independent model 
\cite{Tanaka_Stone_N}, this meant simply $s\in\left[ \frac12,\frac{2}{3} \right]$. However, in the model presented here, 
neither social responsibility (the drive is a threshold-dependent drive) nor viability (the drive is able to invade a wild-type population) 
can be easily and explicitly characterized with analytic methods. Numerical simulations become essential in understanding the interplay
between the two parameters $s$ and $r$.

Our model also shows that a gene drive can achieve complete eradication only if $r$ is small enough. This leads to a particulary
striking conclusion: only very specific choices of $s$ and $r$, corresponding to a small compact region in the $(s,r)$-plane,
can lead to socially responsible and viable eradication gene drives.

More generally, our model shows that the invasion of an eradication drive but also of any gene drive achieving only partial population suppression can be slowed down, stopped or even reversed by the opposing demographic advection term. Thus population dynamics matters for any gene drive affecting population size, even slightly, 
and should be taken into account. We however find that threshold-independent drives can spread spatially even when they lead to population suppression (thereby extending a result found by Beaghton \textit{et al.} \cite{beaghton_gene_2016} for driving-Y to homing-based gene drives): their monostability property remains true, their hair-trigger effect is still observed numerically and the demographic advection term only slows them down but never stops them.
For practical applications, this feature is especially dangerous if one thinks of threshold-independent eradication drives. These should definitely be approached with the highest caution.

Thus, ignoring population dynamics and only focusing on allele frequencies for spatial models of gene drive \cite{Tanaka_Stone_N} 
should be limited to replacement drives, that correspond for instance to the weak selection framework (small values of $s$) or to very
fast population dynamics (large values of $r$). We refer to the discussion in Section \ref{sec:other_models} on the relation between
our model and the two density-independent equations \eqref{eq:TSN} and \eqref{eq:no_dens_dep_at_all}. 

\subsection{Mathematical open problems worthy of attention}

Characterizing analytically the asymptotics $r\to0$ and $r\to+\infty$, namely:
\begin{itemize}
    \item the transition at $s=1/2$ when $r\simeq 0$, from nonexistent traveling waves to existent monostable traveling waves, without 
    intermediate bistable regime;
    \item the vertical asymptote of the level line $\{c_{s,r}=0\}$, its precise value in $[0.65,0.7]$ and its relation with 
    Strugarek--Vauchelet \cite{Strugarek_Vauchelet};
\end{itemize}
are two natural but difficult questions, that are left for further studies.

The main open problem with the systems \eqref{sys:dens_logistic} and \eqref{sys:freq_logistic} is the existence of nontrivial traveling waves.
We briefly explained in Section \ref{sec:analytical_results} why this problem is very difficult: on one hand, the nonexistence result of
Theorem \ref{thm:main_nonexistence} means that there will necessarily be requirements on the parameters $s$ and $r$, and on the other hand
known methods for proofs of existence are likely to be inappropriate, especially in the most interesting parameter range $s\in[1/2,2/3]$.

Some methods to construct traveling waves do not guarantee the monotonicity of the constructed profiles.
Numerically, we only observed monotonic profiles $N$ and $P$, but we did not manage to prove the 
\textit{a priori} monotonicity of these profiles. Thus at this point it remains possible 
that traveling waves with non-monotonic profiles coexist with the observed monotonic ones.

These questions of existence lead to another related crucial issue: the stability properties of traveling waves\footnote{This whole paragraph is intentionally vague regarding the notion of stability, as we do not want to enter into details here. Of course, an unambiguous clarification would be necessary before any attempt at a rigorous proof.}. Indeed, we know that trivial traveling waves always exist.
When a trivial and a nontrivial wave coexist, can we predict theoretically which one is selected? 
The numerical experiments are ambivalent here. On one hand, they lead naturally to the following (rough) conjecture:
\begin{conj}
In the whole parameter region where monotonic nontrivial traveling waves are numerically selected, 
\begin{enumerate}
    \item trivial traveling waves are unstable;
    \item monotonic nontrivial traveling waves exist and are stable.
\end{enumerate}
\end{conj}
But on the other hand, 
it is numerically very challenging to distinguish cases where only trivial waves exist from cases where the two kinds coexist
but the trivial ones are the only stable ones. Hence, at this point, it remains unclear whether the whole bottom right region of 
Figure \ref{fig:heatmap_large_r} corresponds to a nonexistence result or not. We strongly believe that Theorem \ref{thm:main_nonexistence} is
not optimal, but we do not know if we can expect the optimal result to cover the whole region where $c_{s,r}=2\sqrt{r}$.

Finally, it would of course be interesting to improve our theoretical knowledge on the sign of the speed of nontrivial traveling
waves. A direction that might be fruitful is the following: find upper or lower estimates for the decreasing function 
$h:[0,1]\to[0,1]$ defined by $h(P)=N$ and plug these estimates into the Nadin--Strugarek--Vauchelet formula \cite{Nadin_Strugarek_Vauchelet} 
that we used repeatedly in the proofs:
\begin{equation*}
    \operatorname{sign}(c)=\operatorname{sign}\left( -\int_0^1 h(V)^4(r(1-h(V))+1)Vs(1-V)\left(V-\frac{2s-1}{s}\right)\textup{d}V \right).
\end{equation*}
Let us point out that a second similar formula can be obtained by integration by parts of the equation on $N$ multiplied by $N'$:
\begin{equation*}
    \operatorname{sign}(c)=\operatorname{sign}\left( \frac{r}{6}(1-(1-s)n_\star^3)-\int_0^1 s(1-V)h(V)^2\left(r\left(1-\frac{2h(V)}{3}\right)+1\right)\textup{d}V \right),
\end{equation*}
where $n_\star=\max\left(0,1-\frac{s}{r(1-s)}\right)$.
Unfortunately we did not manage to deduce anything interesting from this formula. But, again, it might prove fruitful once estimates on $h$ are
established.

\subsection{Effect of stochasticity}

We checked the robustness of our results by running stochastic simulations of the model given in system~\eqref{sys:dens_logistic}.  
More precisely, our stochastic model is a variant of our deterministic model implemented as a modified Gillespie algorithm, 
with Brownian motion of individuals, numbers of offspring drawn according to a Poisson law and death times drawn according 
to an exponential law. Each simulation is stopped when one allele completely disappears from the spatial domain or when a 
given final time is reached (even if no allele has disappeared yet). Space was discretized into $100$ demes connected by 
dispersal to the nearest neighboring sites, with emigration probability equal to $0.1$; the carrying capacity of a deme is 
$K = 1000$. The simulation codes are currently available online at \url{\urlrepo}.

These simulations confirmed the results of the deterministic version of the model (compare Figure~\ref{fig:stoch} to Figure \ref{fig:heatmap_large_r}).

\begin{figure}
    \centering
    \includegraphics[width = 8cm]{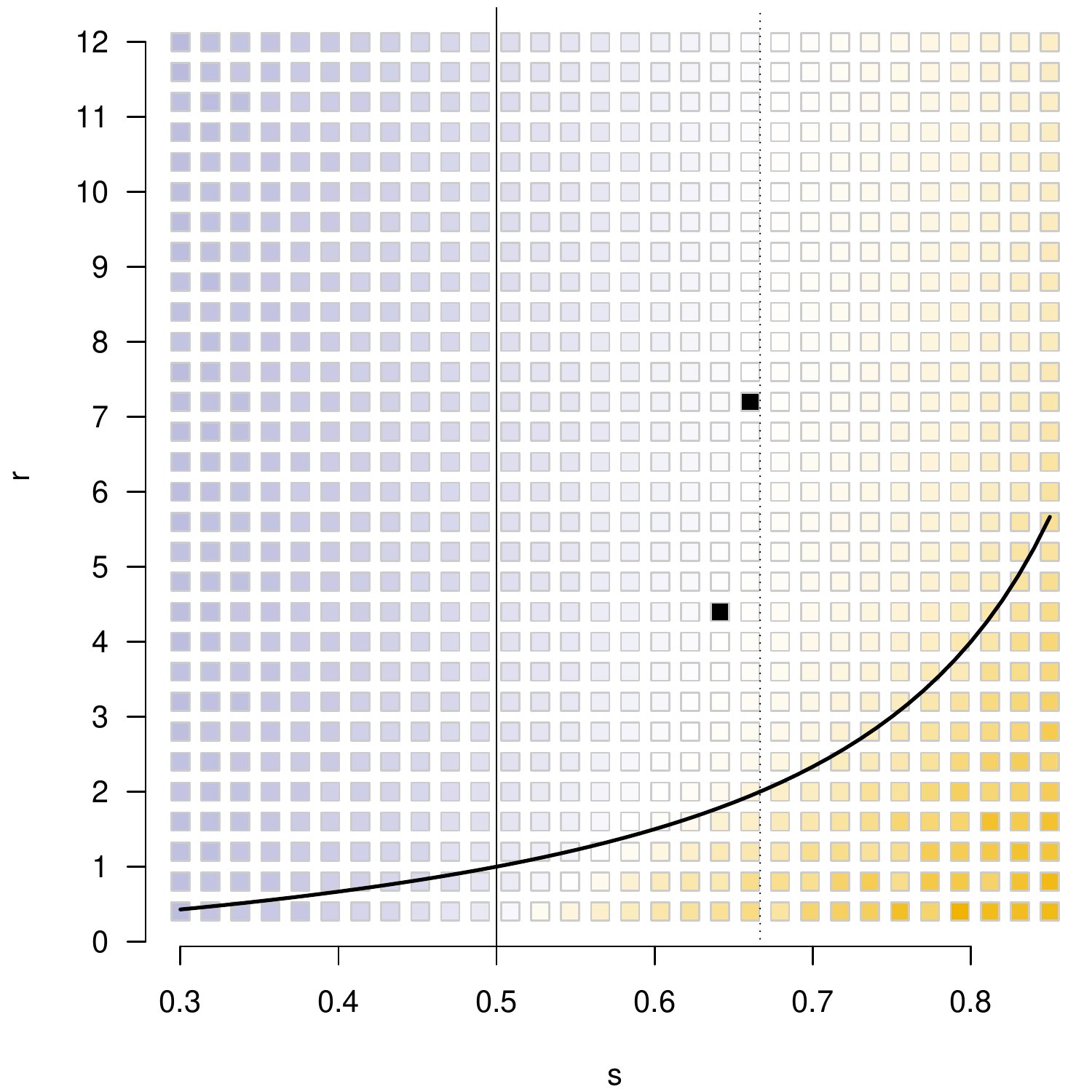}
    \caption{Heatmap in the $(s,r)$ plane for the stochastic simulations. Each point is the color-coded outcome of one stochastic 
    simulation run with the corresponding $(s, r)$ parameters. Blue: the simulation is stopped by the extinction of $O$, yellow: 
    the simulation is stopped by the extinction of $D$, black: no extinction occurs before the final time. The more intense the 
    color, the faster the extinction of the last individual carrying the first lost allele. The thick curve represents the eradication threshold. }
    \label{fig:stoch}
\end{figure}

\subsection{Genericity of this ``wave reversal by opposing demographic advection'' phenomenon}

\subsubsection{Other choices of birth and death rates}
Numerically, several other cases were considered for $B$ and $N$.
\begin{enumerate}
    \item Constant death rate, weak or strong Allee effect on the wild-type dynamics with threshold $a<1$: $D(n)=1$,
	$B(n)=\max\left( r(1-n)(n-a)+1,0 \right)$ (as a birth rate, $B$ has to be nonnegative). 
        The Allee effect is weak if $a\in(-1,0]$ and strong if $a\in(0,1)$ 
	(if $a\leq -1$, there is again no Allee effect). The eradication condition is $r<\frac{4s}{(1-s)(1-a)^2}$.
    \item Constant birth rate, logistic wild-type growth: $B(n)=r+1$, $D(n)=1+rn$. Again, the eradication condition
	is $r<\frac{s}{1-s}$.
    \item Constant birth rate, weak or strong Allee effect with threshold $a\in(-1,1)$ on the wild-type growth: $B(n)=r+1$,
	$D(n)=1+r+r(n-1)(n-a)$. The eradication conditions are $(1-a)^2\leq 4s$ or $r<\frac{4s}{(1-a)^2-4s}$. Note that the first 
	condition does not depend on $r$.
\end{enumerate}

The resulting heatmaps can be observed on Figure \ref{fig:other_B_D}. Interestingly, many features of Figure \ref{fig:heatmap_large_r}
remain true: the $0$-level set is an increasing graph, it is stuck in the parameter range defined by the two thresholds $s=1/2$
and $s=2/3$ of the density-independent equation \eqref{eq:no_dens_dep_at_all}. However, the strong Allee effect on $B$ 
seems to break the strict convexity of the $0$-level set and the strong Allee effect on $D$ seems to simply the picture in the 
sense that all nonviable gene drives satisfy $P=0$.

We point out that many analytical techniques used in the proofs of Theorems \ref{thm:main_nonexistence} and 
\ref{thm:main_sign_when_nontrivial_existence} can be successfully applied to those cases. For the sake of brevity, we do
not give details.

\begin{figure}
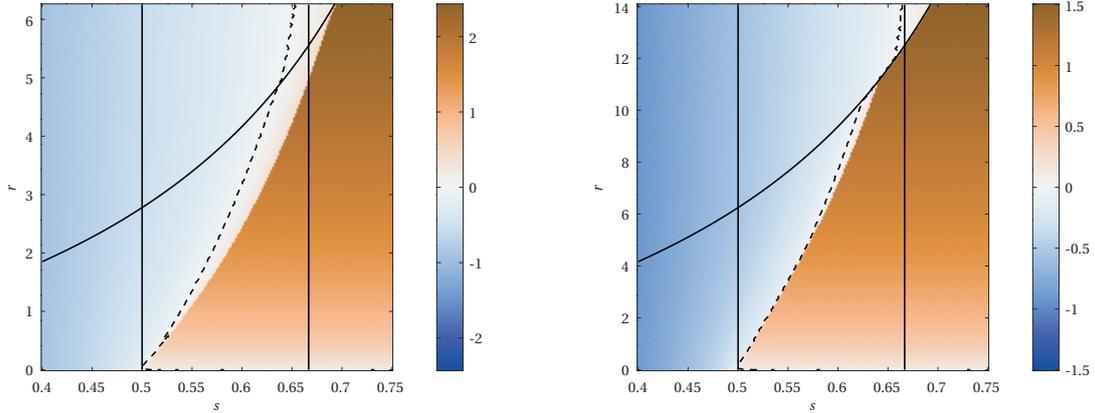
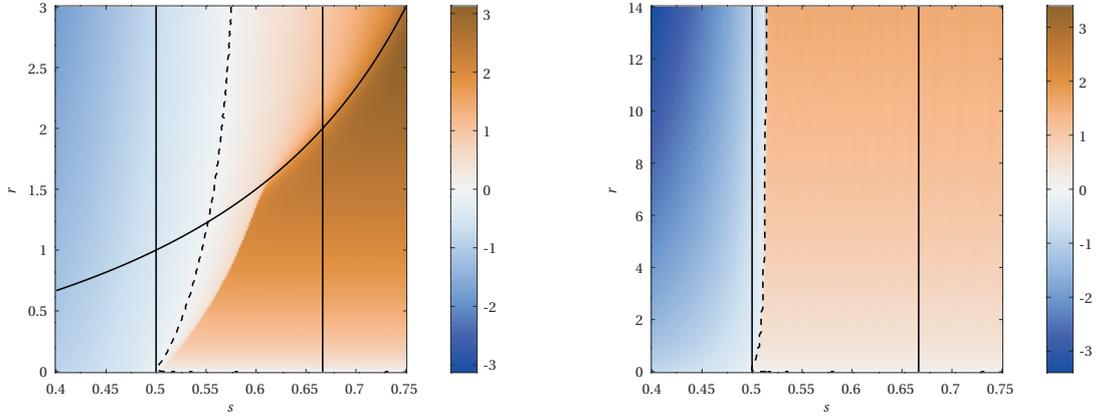

\begin{subfigure}[t]{.49\textwidth}
    \resizebox{\textwidth}{!}{\input{heatmap_Allee_effect_a_-2_10E-1.tex}}
    \caption{Weak Allee effect on $B$ ($a=-0.2$), constant $D$.}
\end{subfigure}
\hfill
\begin{subfigure}[t]{.49\textwidth}
    \resizebox{\textwidth}{!}{\input{heatmap_Allee_effect_bis_a_2_10E-1.tex}}
    \caption{Strong Allee effect on $B$ ($a=0.2$), constant $D$.}
\end{subfigure}

\begin{subfigure}[t]{.49\textwidth}
    \resizebox{\textwidth}{!}{\input{heatmap_D.tex}}
    \caption{Constant $B$, logistic $D$.}
\end{subfigure}
\hfill
\begin{subfigure}[t]{.49\textwidth}
    \resizebox{\textwidth}{!}{\input{heatmap_Allee_effect_a_2_10E-1_D.tex}}
    \caption{Constant $B$, strong Allee effect on $D$ ($a=0.2$). All drives are eradication drives ($4s\geq 1.6>(1-a)^2=0.16$).}
\end{subfigure}
\caption{Heatmaps of $c_{s,r}$ values in the $(s,r)$ plane for different choices of $B$ and $D$ in \eqref{sys:full_densities}.
Dashed curves: $0$-level set.
Strictly convex curves: eradication threshold.
The associated density-independent dynamics are given by \eqref{eq:no_dens_dep_at_all}, with thresholds $s=1/2$ and
$s=2/3$ (see Figure \ref{fig:heatmap_no_dens_dep_at_all}).}
\label{fig:other_B_D}
\end{figure}

\subsubsection{Other gene drive models}

The same numerical method can be applied to different gene drive models: different particular cases of the system
\eqref{sys:initial_full_system}, particular cases of the analogous model obtained when assuming that the gene conversion (homing)
occurs in the germline \cite{Rode_Debarre_2020, noble_evolutionary_2017}, or entirely different models 
\cite{Tanaka_Stone_N, Vella_2017, Deredec_2008, Unckless_2015}. 
We illustrate this possibility with two examples on Figure \ref{fig:other_models}. Again, the main features of 
Figure \ref{fig:heatmap_large_r} are preserved.

However we point out that our analytical tools strongly rely upon the fact that the system is a two-component system: most of them would be unapplicable in a three-component setting. This implies that models where three genotypes $OO$, $OD$ and $DD$ have to be tracked simultaneously cannot be studied this way. Our analytical approach seems to be limited to gene drives with perfect conversion and conversion in the zygote, for which heterozygous individuals can be safely ignored, or to models reduced by means of some approximation (e.g. Hardy-Weinberg proportions as in discrete-time models).

\begin{figure}
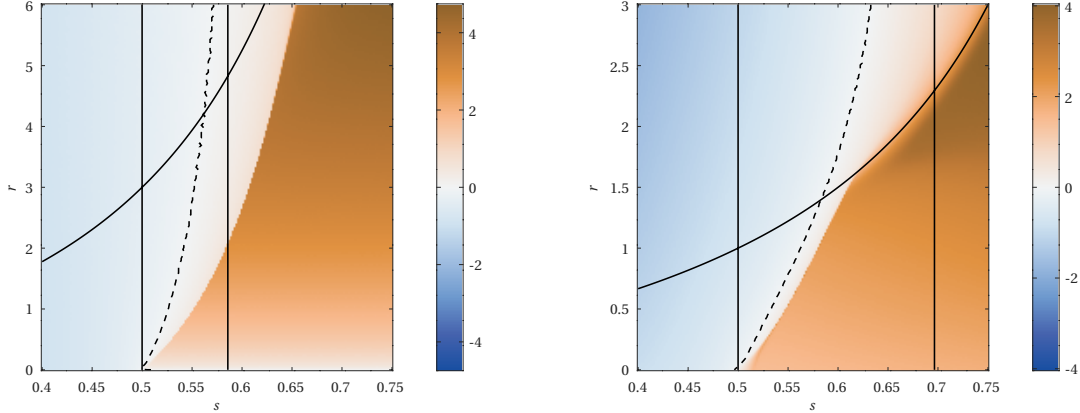

\begin{subfigure}[t]{.49\textwidth}
    \resizebox{\textwidth}{!}{\input{heatmap_fecundity.tex}}
    \caption{In \eqref{sys:dens_logistic}, selection acts on fecundity instead of survival ($\beta_D=1-s$, $\omega_D=1$).}
\end{subfigure}
\hfill
\begin{subfigure}[t]{.49\textwidth}
    \resizebox{\textwidth}{!}{\input{heatmap_freq_TSN.tex}}
    \caption{In \eqref{sys:freq_logistic}, the equation on $p$ is replaced by \eqref{eq:GCD}.}
\end{subfigure}
\caption{Heatmaps of $c_{s,r}$ values in the $(s,r)$ plane for different gene drive models. Dashed curves: $0$-level set.
Strictly convex curves: eradication threshold.
Left vertical line: monostability--bistability threshold for the density-independent dynamics. 
Right vertical line: $0$-level set of the speed of the density-independent dynamics.}
\label{fig:other_models}
\end{figure}

\subsubsection{Extension to a bistable mosquito--Wolbachia model}

Finally, the same ideas can also be applied to different models from mathematical biology that share the common feature
of leading to a two-component bistable reaction--diffusion system where one population with low carrying capacity tries to invade 
another population with high carrying capacity. Indeed, such a model will again involve an opposing advection term of the form 
$-2\nabla\left( \log n \right)\cdot\nabla p$ and it should be clear at this point that this term is the main responsible for
the qualitative features of Figure \ref{fig:heatmap_large_r}.

As an example, we consider a Wolbachia infection in a mosquito population, as described in 
\cite{Barton_Turelli,StrugarekPhD,Nadin_Strugarek_Vauchelet}: 
the infection is transmitted vertically, perfectly, through mothers. The infection reduces fertility of mothers by a 
factor $f_w$. Hatching rate is reduced by a factor $\omega_H$ for crosses involving an infected father and an uninfected mother. All other crossings have the same hatching rates. Assuming equal sex ratios, no sex differences in death rate and no
differences in diffusion rate, we have
\begin{equation*}
\begin{cases}
\frac{\partial n_w}{\partial t}-\Delta n_w = \left( \left(\frac{n_w}{n}\right)^2 f_w + \frac{n_w}{n}\frac{n_s}{n}f_w \right)B(n)n - D(n) n_w \\
\frac{\partial n_s}{\partial t}-\Delta n_s = \left( \left(\frac{n_s}{n}\right)^2 + \frac{n_w}{n}\frac{n_s}{n}\omega_H \right)B(n)n - D(n) n_s,
\end{cases}
\end{equation*}
where $n_w$ and $n_s$ are the densities of infected and uninfected individuals, respectively, and $n = n_w + n_s$ is
the total population density. Rewriting this system in terms of $n$ and $p = \frac{n_w}{n_w + n_s}$, we obtain
\begin{equation*}
\begin{cases}
\frac{\partial n}{\partial t}-\Delta n = \left( 1 - p (2 - f_w - \omega_H) + p^2 (1-\omega_H) \right) B(n) n - D(n) n \\
\frac{\partial p}{\partial t}-\Delta p-2\nabla\left( \log n \right)\cdot\nabla p = B(n)p (1-p) \left(p (1-\omega_H) - (1-f_w)\right) 
\end{cases}
\end{equation*}
The interior equilibrium for $p$, 
\begin{equation}
p^* = \frac{1 - f_w}{1 - \omega_H},
\end{equation}
is admissible (\textit{i.e.}, between $0$ and $1$) when $\omega_H \leq f_w$; it is unstable. 
With this example, we can indeed have a wave towards lower densities 
(a fully infected population is smaller than a fully uninfected population). 

Numerical simulations for this model are displayed on Figure \ref{fig:Wolbachia}. Again, this figure shares many similarities with
Figure \ref{fig:heatmap_large_r}. This illustrates once more that the important term in this whole class of models when concerned about
spreading properties is the opposing advection term.

\begin{figure}
    \resizebox{.49\textwidth}{!}{\input{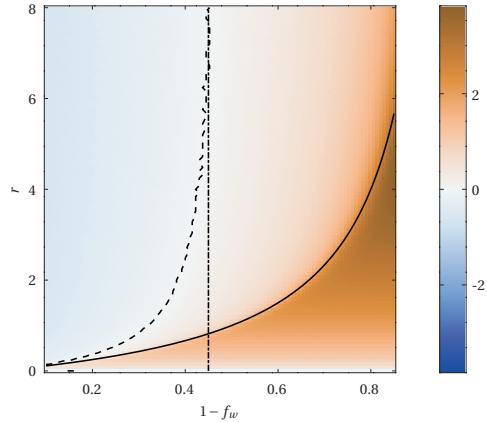}}
    \caption{Heatmaps of $c_{f_{w},r}$ values in the $(1-f_w,r)$ plane for the Wolbachia system with $\omega_H = 0.1$. 
    Dashed curve: $0$-level set.
    Solid curve: level of $r$ below which Wolbachia leads to extinction.
    Dashed-dotted line: $0$-level set of the speed of the density-independent dynamics.}
    \label{fig:Wolbachia}
\end{figure}

\subsection{Biological implications} 
We found that a threshold-independent homing drive would spread spatially, even if it leads to the eradication of the target population. This extends the results of Beaghton \textit{et al.} \cite{beaghton_gene_2016}, obtained by a driving-Y, to another type of gene drive. For threshold-dependent drives, however, we find that the parameter space allowing for spatial spread is more restricted than what is found in models neglecting changes in population size and the resulting opposing demographic advection of individuals \cite{Tanaka_Stone_N}. These simpler models find that a drive spreads as long as the release threshold is sufficiently low (lower than $1/2$ in \cite{barton_dynamics_1979, Barton_Turelli}, lower than a value close to $0.565$ in \cite{Tanaka_Stone_N}) and the initial condition is sufficiently large. In contrast, we find that these conclusions remain true for pure replacement drives but that spatial spread is limited for suppression or eradication drives, all the more as suppression is strong. These conclusions can be extended to other, non-drive, systems for population control such as Wolbachia. 

The parameter range displayed on Figure \ref{fig:heatmaps} and other figures for the Malthusian growth rate, namely $r\in[0,12]$, 
seems to be biologically realistic and relevant. Indeed, we found $r=11.5$ in \cite{White_Rohani_Sait_2010} for \textit{Aedes aegypti}, 
$2\leq r\leq 9.8$ in \cite{Dye_1984} for \textit{Aedes aegypti}, $r=1.4$ in \cite{Mireji_Beier_2020} for 
\textit{Anopheles gambiae}, $r=0.2$ in \cite{Haq_Reisen_Aslamkhan_1981} for \textit{Anopheles stephensi}
\footnote{From \cite{White_Rohani_Sait_2010,Dye_1984,Mireji_Beier_2020,Haq_Reisen_Aslamkhan_1981}, $r$ is deduced by using the 
formula $r=b/d-1$, where $b$ is the constant adult reproduction rate (namely, the birth rate corrected by 
taking into account juvenile mortality), $d$ is the constant adult death rate, and $b$ and $d$ are given in the same unit, say day$^{-1}$.
Compared with the classical formula $r=b-d$, here we divide by $d$ to account for the fact that in system \eqref{sys:dens_logistic} the
time unit is modified so that the death rate is scaled to $1$.}, 
$r=1.3$ in \cite{Roques_Bonnefon_2016} for \textit{Aedes albopictus}.

We have hence characterised a new reason for the failure of spatial spread of suppression drives, in the form of opposing demographic advection. This phenomenon was expected given previous work on spatial dynamics of alleles (as reviewed in \cite{Dhole_Lloyd_Gould_2020}), but we clarify conditions under which it occurs. Other models of spatial spread, and in particular individual-based models, had already identified some reasons why the spatial spread of a suppression drive may fail. If the drive suppresses the local population too much and if the density target population is spatially heterogeneous, the drive may go extinct locally with the eradication of a local subpopulation before it could spread to other locations \cite{north_modelling_2013}. Strategies relying on the eradication of the target population are also limited by the potential recolonization of emptied locations by wild-type individuals \cite{north_modelling_2013, north_modelling_2019, champer_suppression_2019} (such recolonizations can also be observed in our stochastic simulations). Finally, the evolution of resistance to the drive itself, which already hinders the success of gene drives in well-mixed populations \cite{Unckless_2017}, also affects their spatial spread \cite{beaghton_requirements_2017}. 

Our model was derived under limiting assumptions, including a $100\%$ homing rate, and either homing taking place very early in development or the drive being dominant. Gene drives currently being designed in laboratories do not exactly match these assumptions. While we are pessimistic that analytical results can be obtained when these assumptions are relaxed, future numerical or computational (individual-based) studies will be useful to assess the generality of our findings. The results of individual-based simulations of the spatial spread of underdominance gene drive systems \cite{champer_population_2020} are encouraging. The authors indeed found limited spatial spread even in the case of $p^*<1/2$, in simulations initiated with the left half of the arena occupied by drive homozygotes and the right half by wild-type homozygotes (2L2T system; Fig. 3F) -- initial conditions similar to ours.   

To conclude, this work highlights the importance of not neglecting demographic dynamics in evolutionary models when the spread of an allele affect the size of a population. 

\section*{Acknowledgements}
We are grateful to the INRAE MIGALE bioinformatics facility (MIGALE, INRAE, 2020. Migale bioinformatics Facility, doi: 10.15454/1.5572390655343293E12) for providing computing resources. FD is funded by an Agence Nationale de la Recherche JCJC grant TheoGeneDrive ANR-19-CE45-0009-01.

\bibliographystyle{plain}
\bibliography{ref,floRefs}

\begin{thebibliography}{10}

\bibitem{alphey_opinion_2020}
Luke~S. Alphey, Andrea Crisanti, Filippo~(Fil) Randazzo, and Omar~S. Akbari.
\newblock Opinion: {Standardizing} the definition of gene drive.
\newblock {\em Proceedings of the National Academy of Sciences},
  117(49):30864--30867, December 2020.

\bibitem{altrock_using_2010}
Philipp~M. Altrock, Arne Traulsen, R.~Guy Reeves, and Floyd~A. Reed.
\newblock Using underdominance to bi-stably transform local populations.
\newblock {\em Journal of Theoretical Biology}, 267(1):62--75, November 2010.

\bibitem{barton_dynamics_1979}
N~H Barton.
\newblock The dynamics of hybrid zones.
\newblock {\em Heredity}, 43(3):341--359, December 1979.

\bibitem{Barton_Turelli}
Nick~H. Barton and Michael Turelli.
\newblock Spatial waves of advance with bistable dynamics: Cytoplasmic and
  genetic analogues of {A}llee effects.
\newblock {\em The American Naturalist}, 178(3):E48--E75, 2011.
\newblock PMID: 21828986.

\bibitem{beaghton_gene_2016}
Andrea Beaghton, Pantelis~John Beaghton, and Austin Burt.
\newblock Gene drive through a landscape: {Reaction}–diffusion models of
  population suppression and elimination by a sex ratio distorter.
\newblock {\em Theoretical Population Biology}, 108:51--69, April 2016.

\bibitem{beaghton_requirements_2017}
Andrea Beaghton, Andrew Hammond, Tony Nolan, Andrea Crisanti, H.~Charles~J.
  Godfray, and Austin Burt.
\newblock Requirements for {Driving} {Antipathogen} {Effector} {Genes} into
  {Populations} of {Disease} {Vectors} by {Homing}.
\newblock {\em Genetics}, 205(4):1587--1596, April 2017.

\bibitem{burt_site-specific_2003}
Austin Burt.
\newblock Site-specific selfish genes as tools for the control and genetic
  engineering of natural populations.
\newblock {\em Proceedings of the Royal Society of London. Series B: Biological
  Sciences}, 270(1518):921--928, May 2003.

\bibitem{champer_cheating_2016}
Jackson Champer, Anna Buchman, and Omar~S. Akbari.
\newblock Cheating evolution: engineering gene drives to manipulate the fate of
  wild populations.
\newblock {\em Nature Reviews Genetics}, 17(3):146--159, March 2016.

\bibitem{champer_suppression_2019}
Jackson Champer, Isabel Kim, Samuel~E. Champer, Andrew~G. Clark, and Philipp~W.
  Messer.
\newblock Suppression gene drive in continuous space can result in unstable
  persistence of both drive and wild-type alleles.
\newblock preprint, Ecology, September 2019.

\bibitem{champer_population_2020}
Jackson Champer, Joanna Zhao, Samuel~E. Champer, Jingxian Liu, and Philipp~W.
  Messer.
\newblock Population {Dynamics} of {Underdominance} {Gene} {Drive} {Systems} in
  {Continuous} {Space}.
\newblock {\em ACS Synthetic Biology}, 9(4):779--792, April 2020.

\bibitem{Deredec_2008}
Anne Deredec, Austin Burt, and H.~C.~J. Godfray.
\newblock The population genetics of using homing endonuclease genes in vector
  and pest management.
\newblock {\em Genetics}, 179(4):2013--2026, 2008.

\bibitem{Dhole_Lloyd_Gould_2020}
Sumit {Dhole}, Alun~L {Lloyd}, and Fred {Gould}.
\newblock {Gene drive dynamics in natural populations: The importance of
  density dependence, space and sex}.
\newblock {\em Annual review of ecology, evolution, and systematics},
  51(1):505--531, November 2020.

\bibitem{Dye_1984}
Christopher Dye.
\newblock Models for the population dynamics of the yellow fever mosquito,
  {A}edes aegypti.
\newblock {\em Journal of Animal Ecology}, 53(1):247--268, 1984.

\bibitem{Octave}
John~W. Eaton, David Bateman, S{\o}ren Hauberg, and Rik Wehbring.
\newblock {\em {GNU Octave} version 5.1.0 manual: a high-level interactive
  language for numerical computations}, 2019.

\bibitem{efsa_panelon_genetically_modified_organisms_gmo_adequacy_2020}
{EFSA Panel on Genetically Modified Organisms (GMO)}, Hanspeter Naegeli,
  Jean‐Louis Bresson, Tamas Dalmay, Ian~C Dewhurst, Michelle~M Epstein,
  Philippe Guerche, Jan Hejatko, Francisco~J Moreno, Ewen Mullins, Fabien
  Nogué, Nils Rostoks, Jose~J Sánchez~Serrano, Giovanni Savoini, Eve
  Veromann, Fabio Veronesi, Michael~B Bonsall, John Mumford, Ernst~A Wimmer,
  Yann Devos, Konstantinos Paraskevopoulos, and Leslie~G Firbank.
\newblock Adequacy and sufficiency evaluation of existing {EFSA} guidelines for
  the molecular characterisation, environmental risk assessment and
  post‐market environmental monitoring of genetically modified insects
  containing engineered gene drives.
\newblock {\em EFSA Journal}, 18(11), November 2020.

\bibitem{Esvelt_2014}
Kevin~M. Esvelt, Andrea~L. Smidler, Flaminia Catteruccia, and George~M. Church.
\newblock Emerging {T}echnology: {C}oncerning {RNA}-guided gene drives for the
  alteration of wild populations.
\newblock {\em eLife}, 3:e03401, jul 2014.

\bibitem{Fisher_1937}
Ronald~Aylmer Fisher.
\newblock The wave of advance of advantageous genes.
\newblock {\em Annals of eugenics}, 7(4):355--369, 1937.

\bibitem{Girardin_2016_2}
L\'{e}o Girardin.
\newblock Non-cooperative {Fisher{--}KPP} systems: traveling waves and
  long-time behavior.
\newblock {\em Nonlinearity}, 31(1):108, 2018.

\bibitem{Girardin_2019}
L\'{e}o Girardin.
\newblock The effect of random dispersal on competitive exclusion -- {A}
  review.
\newblock {\em Math. Biosci.}, 318:108271, 2019.

\bibitem{Girardin_Calvez_Debarre}
L\'{e}o Girardin, Vincent Calvez, and Florence D\'{e}barre.
\newblock Catch me if you can: A spatial model for a brake-driven gene drive
  reversal.
\newblock {\em Bulletin of Mathematical Biology}, 81(12):5054--5088, Dec 2019.

\bibitem{greenbaum_designing_2019}
Gili Greenbaum, Marcus~W. Feldman, Noah~A. Rosenberg, and Jaehee Kim.
\newblock Designing gene drives to limit spillover to non-target populations.
\newblock preprint, Evolutionary Biology, June 2019.

\bibitem{Haq_Reisen_Aslamkhan_1981}
Nadira Haq, William~K. Reisen, and Muhammad Aslamkhan.
\newblock The effects of {N}osema algerae on the horizontal life table
  attributes of {A}nopheles stephensi under laboratory conditions.
\newblock {\em Journal of Invertebrate Pathology}, 37(3):236 -- 242, 1981.

\bibitem{KPP_1937}
Andrei~N. Kolmogorov, I.~G. Petrovsky, and N.~S. Piskunov.
\newblock {\'Etude} de l'\'equation de la diffusion avec croissance de la
  quantit\'e de mati\`ere et son application \`a un probl\`eme biologique.
\newblock {\em Bulletin Universit\'e d'\'Etat {\`{a}} Moscou}, 1:1--25, 1937.

\bibitem{marshall_confinement_2012}
John~M. Marshall and Bruce~A. Hay.
\newblock Confinement of gene drive systems to local populations: {A}
  comparative analysis.
\newblock {\em Journal of Theoretical Biology}, 294:153--171, February 2012.

\bibitem{Mireji_Beier_2020}
P.~O. Mireji, J.~Keating, A.~Hassanali, C.~M. Mbogo, M.~N. Muturi, J.~I.
  Githure, and J.~C. Beier.
\newblock Biological cost of tolerance to heavy metals in the mosquito
  {A}nopheles gambiae.
\newblock {\em Medical and Veterinary Entomology}, 24(2):101--107, 2010.

\bibitem{Nadin_Strugarek_Vauchelet}
Gr\'{e}goire Nadin, Martin Strugarek, and Nicolas Vauchelet.
\newblock Hindrances to bistable front propagation: application to
  \textit{{W}olbachia} invasion.
\newblock {\em J. Math. Biol.}, 76(6):1489--1533, 2018.

\bibitem{NASEM_2016}
Engineering National Academies~of Sciences and Medicine.
\newblock {\em Gene Drives on the Horizon: Advancing Science, Navigating
  Uncertainty, and Aligning Research with Public Values}.
\newblock The National Academies Press, Washington, DC, 2016.

\bibitem{noble_evolutionary_2017}
Charleston Noble, Jason Olejarz, Kevin~M. Esvelt, George~M. Church, and
  Martin~A. Nowak.
\newblock Evolutionary dynamics of {CRISPR} gene drives.
\newblock {\em Science Advances}, 3(4):e1601964, April 2017.

\bibitem{north_modelling_2013}
Ace North, Austin Burt, and H.~Charles~J. Godfray.
\newblock Modelling the spatial spread of a homing endonuclease gene in a
  mosquito population.
\newblock {\em Journal of Applied Ecology}, 50(5):1216--1225, October 2013.

\bibitem{north_modelling_2019}
Ace~R. North, Austin Burt, and H.~Charles~J. Godfray.
\newblock Modelling the potential of genetic control of malaria mosquitoes at
  national scale.
\newblock {\em BMC Biology}, 17(1):26, December 2019.

\bibitem{Rode_Debarre_2020}
Nicolas~O. Rode, Virginie Courtier-Orgogozo, and Florence D{\'e}barre.
\newblock Can a population targeted by a {CRISPR}-based homing gene drive be
  rescued?
\newblock {\em G3: Genes, Genomes, Genetics}, 10(9):3403--3415, 2020.

\bibitem{Roques_Bonnefon_2016}
Lionel Roques and Olivier Bonnefon.
\newblock Modelling population dynamics in realistic landscapes with linear
  elements: A mechanistic-statistical reaction-diffusion approach.
\newblock {\em PLOS ONE}, 11(3):1--20, 03 2016.

\bibitem{sinkins_gene_2006}
Steven~P. Sinkins and Fred Gould.
\newblock Gene drive systems for insect disease vectors.
\newblock {\em Nature Reviews Genetics}, 7(6):427--435, June 2006.

\bibitem{StrugarekPhD}
Martin Strugarek.
\newblock {\em {Mathematical modeling of population dynamics, applications to
  vector control of {A}edes spp. (Diptera:Culicidae)}}.
\newblock Theses, {Sorbonne Universit{\'e} , UPMC}, September 2018.

\bibitem{Strugarek_Vauchelet}
Martin Strugarek and Nicolas Vauchelet.
\newblock Reduction to a single closed equation for 2-by-2 reaction-diffusion
  systems of {L}otka-{V}olterra type.
\newblock {\em SIAM J. Appl. Math.}, 76(5):2060--2080, 2016.

\bibitem{Tanaka_Stone_N}
Hidenori Tanaka, Howard~A. Stone, and David~R. Nelson.
\newblock Spatial gene drives and pushed genetic waves.
\newblock {\em Proceedings of the National Academy of Sciences},
  114(32):8452--8457, 2017.

\bibitem{Unckless_2017}
Robert~L. Unckless, Andrew~G. Clark, and Philipp~W. Messer.
\newblock Evolution of resistance against crispr/cas9 gene drive.
\newblock {\em Genetics}, 205(2):827--841, 2017.

\bibitem{Unckless_2015}
Robert~L. Unckless, Philipp~W. Messer, Tim Connallon, and Andrew~G. Clark.
\newblock Modeling the manipulation of natural populations by the mutagenic
  chain reaction.
\newblock {\em Genetics}, 201(2):425--431, 2015.

\bibitem{Vella_2017}
Michael~R. Vella, Christian~E. Gunning, Alun~L. Lloyd, and Fred Gould.
\newblock Evaluating strategies for reversing {CRISPR-Cas9} gene drives.
\newblock {\em Scientific Reports}, 7(1):11038, 2017.

\bibitem{werren_selfish_2011}
J.~H. Werren.
\newblock Selfish genetic elements, genetic conflict, and evolutionary
  innovation.
\newblock {\em Proceedings of the National Academy of Sciences},
  108(Supplement\_2):10863--10870, June 2011.

\bibitem{White_Rohani_Sait_2010}
Steven~M. White, Pejman Rohani, and Steven~M. Sait.
\newblock Modelling pulsed releases for sterile insect techniques: fitness
  costs of sterile and transgenic males and the effects on mosquito dynamics.
\newblock {\em Journal of Applied Ecology}, 47(6):1329--1339, 2010.

\end{thebibliography}

\end{document}